\documentclass[11pt]{siamart0516}
\usepackage{graphicx}        
\usepackage{amssymb,amsmath}
\usepackage{amsmath}
\usepackage{mathtools}
\usepackage{tikz}
\usepackage{pgfplots}
\usepackage{nicefrac}
 \usetikzlibrary{external}
\usepackage{xfrac}
\usepackage[software,hardware]{mymacros}
\usepackage{algorithm} 
\usepackage{algpseudocode}
\usetikzlibrary{matrix}
\usepackage{epstopdf}
\usepackage{mathrsfs}
\usepackage{rotating}

\newcommand{\bbm}{\begin{bmatrix}}
\newcommand{\ebm}{\end{bmatrix}}
\usepgfplotslibrary{groupplots}
\usepackage{pgfplots}
\usepackage{pgfplotstable}
\usepackage{cleveref}
\usepackage{mathrsfs}
\usepackage{caption}
\usepackage{subcaption}
\usepackage{url}
\usepackage{enumitem}
\usepackage{soul}

\newcommand{\chC}{\mathcal \hC}
\newcommand{\chK}{\mathcal \hK}
\newcommand{\chB}{\mathcal \hB}

\newenvironment{customlegend}[1][]{%
    \begingroup
    \csname pgfplots@init@cleared@structures\endcsname
    \pgfplotsset{#1}%
}{  \csname pgfplots@createlegend\endcsname
    \endgroup
}

\def\addlegendimage{\csname pgfplots@addlegendimage\endcsname}

\usetikzlibrary{external}
\usetikzlibrary{calc,patterns,
                 decorations.pathmorphing,
                 decorations.markings}
\newlength\fheight
\newlength\fwidth

\newcommand{\TheTitle}{Identification of Dominant Subspaces for Linear Structured Parametric  Systems and Model Reduction}
\newcommand{\TheShortTitle}{Identification of Dominant Subspaces  and MOR}
\newcommand{\TheAuthors}{Peter Benner, Pawan Goyal, and Igor Pontes Duff}

\headers{\TheShortTitle}{\TheAuthors}

\title{{\TheTitle}\thanks{Submitted to the editors \today.
}}

\author{
	  Peter Benner\thanks{Max Planck Institute for Dynamics of Complex Technical Systems, Sandtorstra\ss e 1, 39106 Magdeburg, Germany; (\email{benner@mpi-magdeburg.mpg.de}).}
	  \and
	  Pawan Goyal\thanks{Corresponding author. Max Planck Institute for Dynamics of Complex Technical Systems, Sandtorstra\ss e 1, 39106 Magdeburg, Germany; (\email{goyalp@mpi-magdeburg.mpg.de}).}
	  \and 
	    Igor Pontes Duff\thanks{Max Planck Institute for Dynamics of Complex Technical Systems, Sandtorstra\ss e 1, 39106 Magdeburg, Germany; (\email{pontes@mpi-magdeburg.mpg.de}).}
      }

\numberwithin{equation}{section}
\numberwithin{theorem}{section}
\numberwithin{figure}{section} 
\numberwithin{algorithm}{section}
\newtheorem{remark}{Remark}[section]
\newtheorem{conject}{Conjecture}[section]
\newtheorem{assump}{Assumption}[section]

\usepackage{soul}
\usepackage{scalerel}

\begin{document}

\maketitle
\begin{abstract}
In this paper, we discuss a novel model reduction framework for generalized linear systems. The transfer functions of these systems are assumed to have a special structure, e.g., coming from second-order linear systems and time-delay systems, and they may also have parameter dependencies. Firstly, we investigate the connection between classic interpolation-based model reduction methods with the reachability and observability subspaces of linear structured parametric systems. We show that if enough interpolation points are taken, the projection matrices of interpolation-based model reduction encode these subspaces. As a result,  we are able to identify the dominant reachable and observable subspaces of the underlying system. Based on this, we propose a new model reduction algorithm combining these features leading to reduced-order systems.  Furthermore, we pay special attention to computational aspects of the approach and discuss its applicability to a large-scale setting. We illustrate the efficiency of the proposed approach with several numerical large-scale benchmark examples. 
\end{abstract}


\section{Introduction}
In this paper, we consider linear structured  parametric systems, whose transfer functions are of the form:
\begin{equation}\label{eq:gen_trans}
\bH(s,\bp) = \mathcal C(s,\bp) \mathcal K(s,\bp)^{-1} \mathcal B(s,\bp),
\end{equation}
where 
\begin{equation}\label{eq:sys_matrices}
\begin{aligned}
\mathcal C(s,\bp) &= \sum_{i=1}^k\gamma_i(s,\bp)\bC_i,&
\mathcal K(s,\bp) &= \sum_{i=1}^l\kappa_i(s,\bp)\bA_i,&
\mathcal B(s,\bp) &= \sum_{i=1}^m\beta_i(s,\bp)\bB_i,
\end{aligned}
\end{equation}
in which $\bA_i \in \Rnn$, $\bB_i \in \Rnm$, $\bC_i\in \Rpn$ are constant matrices, $s$ take values on the imaginary axis, and $\bp = \bbm p^{(1)},\ldots,p^{(d)} \ebm \in \Omega^d$ are the system parameters. $\kappa_i(s,\bp), \beta_i(s,\bp)$ and  $\gamma_i(s,\bp)$ are functions of $s \in \C$ and $\bp\in \Omega^d$. Additionally, the restrictions $\kappa_i(\cdot,\bp),  \beta_i(\cdot,\bp)$ and  $\gamma_i(\cdot,\bp)$ are assumed to be meromorphic functions.  In this paper, we also assume that $\bH(s,\bp)$ is a strictly proper function for all parameters, i.e., $\lim\limits_{s\rightarrow  \pm i\infty }\bH(s,\bp) = 0$. 
The system~\eqref{eq:gen_trans} covers a large class of linear systems, arising in various science and engineering applications, e.g., classical linear systems, second-order systems, time-delay systems, integro-differential systems, and their parameter-dependent variants. 

In order to illustrate the class of systems \eqref{eq:gen_trans}, we consider a dynamical system arising in computational electro-magnetics presented in \cite{morFenB10}.  A discretized system can be obtained by the spatial discretization of electromagnetic field equations, describing the electro-dynamical behavior of microwave devices when the surface losses are included in a physical model. The transfer function of the system has a very particular structure; precisely, it has a fractional integrator and takes the form:
\begin{equation}\label{eq:elect_TF}
\bH(s) = \sqrt{s}\bB^T\left(s^2\bI - \frac{1}{\sqrt{s}}\bD +\bA\right)^{-1}\sqrt{s}\bB.
\end{equation}
If the above equation is compared with the form given in \cref{eq:gen_trans} for a fixed parameter, then the matrices $\cC(s), \cK(s)$ and $\cB(s)$ can be given by
\[ \cC(s) = \sqrt{s}\bB^T,  \quad \cK(s) = \left(s^2\bI - \frac{1}{\sqrt{s}}\bD +\bA\right), \quad ~\, ~~\cB(s) = \sqrt{s}\bB,  \]
and 
\begin{align*}
\gamma_1(s) = \beta_1(s) = \sqrt{s}, \quad   
\kappa_1(s) = s^2, \quad& \kappa_2(s) = -\frac{1}{\sqrt{s}}, \quad \kappa_3(s) \equiv 1,  \\ \bB_1 = \bC_1^T = \bB, \quad  \bA_1 = \bI, \quad&  \bA_2 = \bD \quad  \text{and}~\, \bA_3 = \bA.
\end{align*}
Hence, the transfer function \eqref{eq:elect_TF} fits into our framework \eqref{eq:gen_trans}.

Model order reduction (MOR) has been studied extensively in the literature for some classes of linear systems, see e.g.,~\cite{morAnt05,morBenMS05} for standard linear systems, \cite{morBenS11,morChaLVD06,morEidSLetal07,morReiS08} for second-order systems, \cite{jarlebring2013model,plischke2005transient} for time-delay systems, and the review paper  \cite{BenGS15} for parametric systems.  Furthermore, MOR techniques have been investigated by several researchers for the generalized linear systems \eqref{eq:gen_trans}. For a fixed parameter, balanced truncation has been proposed in \cite{breiten2016structure}. The method requires solving the system Gramians, namely reachability and observability Gramians, which can be  a computationally challenging task in a large-scale setting. Another popular MOR method, transfer function interpolation, has also been studied   \cite{morAntBG10,morBeaG09}, where for a given set of interpolation points, it is shown how to construct an interpolating reduced-order system while preserving the system structure. However, \cite{morAntBG10,morBeaG09} leave an important open problem  about the choice of a good set of interpolation points. Furthermore, we would like to mention that a data-driven approach for structured non-parametric systems has been  studied in \cite{schulze2018data}. Nevertheless, the construction  of the structured reduced-order system is not a straightforward task; more importantly, it is not clear how to construct a reachable and observable system.  

In this paper, we discuss the connection between interpolation-based MOR methods with the reachable and observable subspaces of linear structured parametric systems. We show that if enough interpolation points are taken, the projection matrices of interpolation-based model reduction encode these subspaces.  As a consequence, we propose an approach to construct reduced-order systems preserving the common subspaces containing the most reachable as well as the most observable states. This approach can be seen as a combination of the interpolation-based method in \cite{morAntBG10} and some inspiration from the Loewner framework for first-order systems \cite{morMayA07}.  

The precise structure of the paper is as follows. 
In the subsequent section, we discuss the construction of interpolating reduced-order systems for \eqref{eq:gen_trans} for a given set of interpolation points $s$ and  parameters $\bp$. Thereafter, in \Cref{sec:ContObserv} we define the concepts of reachability and observability for  linear structured parametric systems and connect them with the interpolation-based MOR methods. Subsequently, by  combining both features, we discuss the construction of reduced-order systems keeping the subspaces of the most reachable and observable states simultaneously. In \Cref{sec:numerical}, we illustrate the efficiency of the approach using several benchmark examples and finally conclude with future avenues.  In the rest of the paper, we make use of the following notation.
\begin{itemize}
\item $\texttt{svd}\{\cdot\}$ denotes the singular value decomposition (SVD) of a matrix. 
\item By using \matlab~notation, we denote the first $l$ columns of a matrix $\bV$ by  $\bV(:,1 : l)$.
\item Let $\cV$ be a subspace of $\C^n$ (or $\R^n$).  We denote the orthogonal subspace of $\cV$ by $\cV^{\perp}$. 
\end{itemize}

\section{Preliminary Work}\label{sec:Background}
In this section, we briefly recap an interpolatory framework to construct reduced-order systems. Let us consider  linear systems whose input-output mappings (transfer functions) are given in \eqref{eq:gen_trans}.  Our goal is to construct reduced-order systems having a similar structure using Petrov-Galerkin projection as follows:
\begin{equation}\label{eq:gen_parametric_TF}
\hbH(s,\bp) = \hat{\cC}(s,\bp) \mathcal{\hK}(s,\bp)^{-1} \mathcal{\hB}(s,\bp),
\end{equation}
where 
\begin{align}\label{eq:reduced_generalized}
\mathcal \hC(s,\bp) &= \cC(s,\bp)\bV,&
\mathcal \hK(s,\bp) &= \bW^T \cK(s,\bp)\bV,&
\mathcal \hB(s,\bp) &= \bW^T \bB(s,\bp)\bV.
\end{align}
We aim at determining the matrices $\bV$ and $\bW$ in a way that the resulting reduced-order system interpolates a given set of interpolation points for $s$ and $\bp$. This problem was considered in \cite{morAntBG10}, where the idea of constructing interpolatory non-parametric structured systems \cite{morBeaG09} and classical linear parametric systems \cite{morBauBBetal11} was extended to the systems \eqref{eq:gen_parametric_TF}. However, the authors in \cite{morAntBG10} present the framework for linear structured parametric systems, where the functions given in \eqref{eq:sys_matrices} can be decomposed as $q(s,\bp) := q_s(s)q_p(\bp)$. In the following theorem, we present a variation of this result where this decomposition is not required. 
\begin{theorem}\label{theo:InterpParStruct}
	Let $\bH(s,\bp)$ be a transfer function as in \eqref{eq:gen_trans}. Consider interpolation points $\{\sigma_i,\bp^r_{i}\}$ and $\{\mu_i,\bp^l_{i}\}$, $i \in \{1,\ldots,r\}$, such that $\cK(s,\bp)$ is invertible for $\{s,\bp\} \in \{\sigma_i,\bp_i^r\} \cup \{\mu_i,\bp_i^l\}$.  Furthermore, let the projection matrices $\bV, \bW \in \C^{n\times r}$ be as follows:
	\begin{subequations}\label{eq:ProjMatrices}
		\begin{align}
		\underset{i= 1,\ldots,r}{\spann} \left\{ \mathcal K(\sigma_i,\bp^l_{i}) ^{-1}\mathcal B(\sigma_i,\bp^r_{i})\right\} &\subseteq 	\range{\bV},\\
		\underset{i= 1,\ldots,r}{\spann} \left\{ \mathcal K(\mu_i,\bp^r_{i}) ^{-T}\mathcal C(\sigma_i,\bp^l_{i})^T\right\} &\subseteq 	\range{\bW}.
		\end{align}
	\end{subequations}
	If the reduced matrices are computed as shown in \eqref{eq:reduced_generalized}, then the following conditions are satisfied, for $i =1, \dots, r$:
	\begin{subequations}\label{eq:interpolationGenCond}
		\begin{align}
		\bH(\sigma_i,\bp^l_{i}) &= \hbH(\sigma_i,\bp^l_{i}),\label{eq:interpolationGenCond1}\\
		\bH(\mu_i,\bp^r_{i})    &= \hbH(\mu_i,\bp^r_{i}).\label{eq:interpolationGenCond2}
		\end{align}
	\end{subequations}
	Moreover, if $\mu_i = \sigma_i$, and $\bp^{l}_i = \bp^{r}_i =: \bp_i$, for $i \in \{1,\ldots,r\}$, and  $\bH(s,\bp)$ and $\hbH(s,\bp)$ are differentiable at $\sigma_i$ and $\bp_i$ for $i\in \{1,\ldots,r\}$, then along with \eqref{eq:interpolationGenCond}, the following conditions are satisfied:
	
	\begin{subequations}\label{eq:deri_info}
		\begin{align}
		\dfrac{d}{ds}\bH(\sigma_i,\bp_i) 	&= \dfrac{d}{ds}\hbH(\sigma_i,p_i),\label{eq:deri_info1}\\
		\nabla_\bp\bH(\sigma_i,\bp_i) 		&= \nabla_\bp\hbH(\sigma_i,\bp_i).\label{eq:deri_info2}
		\end{align}
	\end{subequations}
\end{theorem}
\begin{proof}
	The theorem can be proven exactly along the lines given in \cite{morAntBG10}. We begin by proving \eqref{eq:interpolationGenCond1}. We have
	\begin{align*}
	\hat \bH(\sigma_i,\bp^l_{i}) &= \hat \cC(\sigma_i,\bp^l_{i})\hat \cK(\sigma_i,\bp^l_{i})^{-1}\hat \cB(\sigma_i,\bp^l_{i})\\
	& =  \cC(\sigma_i,\bp^l_{i}) \bV\hat \cK(\sigma_i,\bp^l_{i})^{-1}\bW^T \cB(\sigma_i,\bp^l_{i}) \\ 
	& =  \cC(\sigma_i,\bp^l_{i}) \bV\hat \cK(\sigma_i,\bp^l_{i})^{-1}\bW^T \cK(\sigma_i,\bp^l_{i}) \underbrace{ \cK(\sigma_i,\bp^l_{i}) ^{-1} \cB(\sigma_i,\bp^l_{i})}_{\in \range{\bV}}\\
	& =  \cC(\sigma_i,\bp^l_{i}) \bV\hat \cK(\sigma_i,\bp^l_{i})^{-1}\bW^T \cK(\sigma_i,\bp^l_{i}) \bV z\\
	& = \cC(\sigma_i,\bp^l_{i}) \bV z
	= \cC(\sigma_i,\bp^l_{i}) \cK(\sigma_i,\bp^l_{i}) ^{-1} \cB(\sigma_i,\bp^l_{i}) = \bH(\sigma_i,\bp^l_{i}).
	\end{align*}
	Analogously, we can prove \eqref{eq:interpolationGenCond2}. Furthermore, we can prove \eqref{eq:deri_info1} and \eqref{eq:deri_info2} along the lines given in \cite{morBeaG09} and \cite{morBauBBetal11}, respectively. For the sake of brevity, we refrain from providing a complete proof. 
\end{proof}

In the previous theorem, we have seen how an interpolatory reduced-order system can be constructed for a given set of interpolation points. However, a good choice of interpolation points for both frequency $(s)$ and the parameters $(\bp)$ remains an open question. Therefore, in this work, we show that for enough interpolation points, we can determine the important subspaces, leading to a good quality of reduced-order systems. For this, in what follows, we first discuss the concepts of reachability and observability for dynamical systems.

\section{Reachability, Observability and Reduced-Order Systems}\label{sec:ContObserv}

This section aims at showing the connection of the Petrov-Galerkin projection matrices $\bV$ and $\bW$ in \eqref{eq:ProjMatrices} (\cref{theo:InterpParStruct}) with the classical concepts of reachability and observability of dynamical systems. Based on this, we can identify the states that are simultaneously least reachable and least observable. This leads us to an algorithm which is a combination of interpolation and SVD techniques, enabling us to construct reduced-order systems for structured parametric systems. We begin here by briefly revisiting some results for first-order linear systems.   

\subsection{Background on First-order Systems}
The transfer function of a first-order system  is given by
\begin{equation}\label{eq:FirstOrderTF}
\bH_{\texttt{fo}}(s) =  \bC\left(s\bI - \bA\right)^{-1}\bB, \quad \text{with~\,~  $\bA \in \C^{n \times n}$,  $\bB \in \C^{n \times m}$ and $\bC \in \C^{p \times n}$.}
\end{equation}
We note that the reachable subspace $\cV_{\scaleto{{R}}{3.25pt}}$ and the observable subspace $\cW_{\scaleto{{O}}{3.25pt}}$ of the system~\eqref{eq:FirstOrderTF} are given  by the smallest subspaces of $\C^n$ such that 
\[ e^{\bA t}\bB \in \cV_{\scaleto{{R}}{3.25pt}}\quad \text{and} \quad   e^{\bA^T t}\bC^T \in \cW_{\scaleto{{O}}{3.25pt}}~\,\text{for every $t\geq 0$.}\]
This essentially follows from standard proofs relating reachability and observability of linear-time invariant systems with the rank of the Kalman reachability and observability matrices defined as follows:
\begin{subequations} \label{eq:KalmanMatrices}
\begin{align} 
\bM_{\scaleto{{R}}{3.25pt}}(\bA,\bB) &= \begin{bmatrix}
\bB &\bA\bB &\bA^2\bB  & \dots & \bA^{n-1}\bB
\end{bmatrix},\,~\text{and} \\ ~\,
\bM_{\scaleto{{O}}{3.25pt}}^T(\bC,\bA) &= \begin{bmatrix}
\bC^T & \bA^T\bC^T &(\bA^2)^T\bC^T  & \dots & (\bA^{n-1})^T\bC^T
\end{bmatrix},
\end{align}
\end{subequations}
see \cite[Chapter 2]{dullerud2013course} for more details. The unreachable subspace, which is the orthogonal complement of $\cV_{\scaleto{{R}}{3.25pt}}$ and denoted by $\cV_{\scaleto{{R}}{3.25pt}}^{\perp}$, consists of the states $\bq_{ur}\in \C^n$  such that $\bq_{ur}^Te^{\bA t}\bB =0$ for every $t\geq 0$. Similarly, the unobservable subspace, characterized by $\cW_{\scaleto{{O}}{3.25pt}}^{\perp}$, consists of the states $\bq_{uo}\in \C^n$ such that $\bC e^{\bA t}\bq_{uo} =0$, for every $t\geq 0$.  By means of the Laplace transform, the reachability and observability subspaces can be seen as  the smallest subspaces of $\C^n$ such that
 \begin{equation}\label{eq:CnContrFirstOrder} 
 (s\bI-\bA)^{-1}\bB \in \cV_{\scaleto{{R}}{3.25pt}}, \quad \text{and} \quad  (s\bI-\bA)^{-T}\bC^T \in \cW_{\scaleto{{O}}{3.25pt}}, ~\,\text{for every $s \in i\R$.}
 \end{equation}
 Additionally, we note that the system \eqref{eq:FirstOrderTF} is reachable if  $\cV_{\scaleto{{R}}{3.25pt}} = \C^n$, and observable if $\cW_{\scaleto{{O}}{3.25pt}} = \C^n$.   A classical result in system theory is that if the system \eqref{eq:FirstOrderTF} is not reachable or not observable, then there exists a system of lower-order, having the same transfer function as the original one. Indeed, the unreachable or unobservable states,  denoted by $\cV_{\scaleto{{R}}{3.25pt}}^{\perp}$ and  $\cW_{\scaleto{{O}}{3.25pt}}^{\perp}$, respectively, can be removed from the dynamics, without changing the transfer function. Moreover, if the system is reachable and observable, then it is minimal, i.e., there exists no lower-order realization for the original transfer function, see \cite[Chapter 2]{dullerud2013course}.
 
 For first-order systems,  a well-known characterization of reachable and observable spaces are given by $\cV_{\scaleto{{R}}{3.25pt}} = \range{\bM_{\scaleto{{R}}{3.25pt}}(\bA,\bB)}$ and $\cW_{\scaleto{{O}}{3.25pt}} = \range{\bM_{\scaleto{{O}}{3.25pt}}^T(\bC,\bA)},$  where $\bM_{\scaleto{{R}}{3.25pt}}(\bA,\bB)$ and $\bM_{\scaleto{{O}}{3.25pt}}^T(\bC,\bA)$ are, respectively, the Kalman reachability and observability matrices given in \eqref{eq:KalmanMatrices}.
Furthermore, the authors in \cite{morAndA90} provided a different characterization of these subspaces which is closely related to the interpolation problem. In particular, they have  shown that
$\cV_{\scaleto{{R}}{3.25pt}} = \range{\bV}$ and  $\cW_{\scaleto{{O}}{3.25pt}} = \range{\bW}$, where
\begin{align*} 
\bV &= \begin{bmatrix}(\sigma_1\bI -\bA)^{-1}\bB &(\sigma_2\bI -\bA)^{-1}\bB  & \dots & (\sigma_{N}\bI -\bA)^{-1}\bB \end{bmatrix}, ~\text{and} 
\\ 
\bW &= \begin{bmatrix} (\sigma_1\bI -\bA)^{-T}\bC^T &(\sigma_2\bI -\bA)^{-T}\bC^T  & \dots & (\sigma_{N}\bI -\bA)^{-T}\bC^T \end{bmatrix}
\end{align*}
with $N\geq n$ and  $\sigma_k \in i\R \cap \Lambda(\bA), k\in \{1,\ldots, N\}$, are distinct points, where $\Lambda(\cdot)$ denotes the spectrum of a matrix. 
Notice that the  above matrices $\bV$ and $\bW$ are the particular matrices $\bV$ and $\bW$ from \Cref{theo:InterpParStruct} when the original system is first-order non-parametric. Moreover, also in  \cite{morAndA90} and more recently in \cite{morMayA07}, the authors have shown that the matrix $\begin{bmatrix} \bW^T\bV & \bW^T\bA\bV \end{bmatrix}$ encodes the dimension of the minimal order realization of the original system, i.e.,
\[\rank{\begin{bmatrix}
	\bW^T\bV & \bW^T\bA\bV
	\end{bmatrix}} = \left\{\begin{array}{l}
\text{order of the minimal realization obtained by} \\
\text{removing unreachable and unobservable states}
\end{array}\right. \]
Furthermore, in \cite{morMayA07}, an algorithm based on the SVD is provided in order to get the minimal realization or reduced-order systems by Petrov-Galerkin projections. 

Inspired by the above discussion on first-order systems, in what follows, we first extend some of the results for linear structured systems \eqref{eq:gen_trans}, and propose an algorithm, allowing us to construct reduced-order systems by removing unreachable and unobservable subspaces.

\subsection{Reachability, Observability, and Reduced Systems for Linear Structured Systems}

As mentioned earlier, unreachable or unobservable states can be removed from first-order systems~\eqref{eq:FirstOrderTF} such that the resulting lower-order system has the same transfer function as the original one. In this section, we study how these notions can be extended to the class of structured systems \eqref{eq:gen_trans}. For the clarity of exposition, we start our discussion with the non-parametric case, i.e.,  linear structured systems of the form 
\begin{equation}\label{eq:gen_trans_nonpar}
\bH(s) = \mathcal C(s) \mathcal K(s)^{-1} \mathcal B(s).
\end{equation}

Dynamical systems can be represented using different frameworks, e.g., as a semi-group in a Hilbert space, as a linear system over a ring of operators and as a functional differential equation.  Depending on the nature of representation or the considered problem, there exist various definitions for reachability and observability. 

For the infinite-dimensional control community, dynamical systems are represented in the infinite-dimensional setting, where the state-space might be seen as an (infinite-dimensional) Hilbert space associated with a semi-group. There, the notions of exact and approximate reachability and observability play an important role, see, e.g., \cite{fuhrmann1976exact, vidyasagar1970controllability}  and \cite[Chapter 4]{curtain2012introduction}. However, in this setup, the dynamical systems are no longer seen as functional differential equations. Hence, discretizing the infinite-dimensional Hilbert space leads to finite-dimensional first-order reduced-order systems that do not preserve the structure in  \eqref{eq:gen_trans}.  

Another point of view is the one in algebraic system theory, which considers dynamical systems as linear systems over a ring, see, e.g., \cite{kamen1978, kamen1991linear}. In this setting, the reachability and observability concepts rely on a rank condition over a ring, e.g., strong and weak reachability/observability, see \cite{lee1981observability,morse1976ring,sontag1976linear}.  However, since these concepts  depend on the underlying ring, we are not aware of a way  to adapt it to a structure-preserving reduction in the Petrov-Galerkin projection framework.

In this paper, we define a weaker notion of reachability/observability that relies only on linear algebra concepts  and the matrices of the realization of the system \eqref{eq:gen_trans_nonpar}. These concepts are related to the realization of the transfer function \eqref{eq:gen_trans_nonpar} as a functional differential equation, see \cite{breiten2016structure,gripenberg1990volterra}.   

\begin{definition}\label{def:Cncontr} The $\C^n$--reachable subspace $\cV_{\scaleto{{R}}{3.25pt}}$ associated to the pair of functions $(\cK(s), \cB(s))$ is the smallest subspace of $\C^n$ which contains $\cK(s)^{-1}\cB(s)$ for all $s \in i\R$. In other words, if $\cV_{\scaleto{{R}}{3.25pt}} = \range{\bV_{\scaleto{{R}}{3.25pt}}}$,  where $\bV_{\scaleto{{R}}{3.25pt}}\in \C^{n \times r_{\scaleto{{R}}{3.25pt}}}$ is a full column rank matrix and $r_{\scaleto{{R}}{3.25pt}} \leq n$, we have
\begin{equation}\label{eq:Cncontr}
	\cK(s)^{-1}\cB(s) = \bV_{\scaleto{{R}}{3.25pt}} y_{\scaleto{{R}}{3.25pt}}(s),
\end{equation}
in which  $y_{\scaleto{{R}}{3.25pt}}(s)\in \C^{r_{\scaleto{{R}}{3.25pt}} \times m}$ is a matrix of meromorphic function. We say that $(\cK(s), \cB(s))$  is $\C^n$--reachable if $\cV_{\scaleto{{R}}{3.25pt}} = \C^n$. 
\end{definition}
The $\C^n$--unreachable subspace consists of the states $\bq_{ur} \in \C^n$  such that \[\bq_{ur}^T\cK(s)^{-1}\cB(s) =0\quad \forall s \in i\R. \] As a consequence, it is characterized by $\cV_{\scaleto{{R}}{3.25pt}}^{\perp}$.  It is worth noting that, in the case of first-order systems \eqref{eq:FirstOrderTF}, $\cK(s)^{-1}\cB(s) = (s\bI-\bA)^{-1}\bB$ and the above definition is a natural extension of \eqref{eq:CnContrFirstOrder}.  

We recall that the transfer between the input to the state for system \eqref{eq:gen_trans_nonpar} is given by 
	\[\bx(s) = \cK(s)^{-1}\cB(s)\bu(s). \]
Hence, 	the $\C^n$--reachable subspace $\cV_{\scaleto{{R}}{3.25pt}}$ correspond to the smallest subspace that contains the states $\bx(s)$ for every $s$. This justifies the use of the reachability terminology.

To connect the frequency and time-domains, let $\bM(t)$ denote the inverse Laplace transform of $\cK(s)^{-1}\cB(s)$. Hence, by definition $\bM(t) = \bV_{\scaleto{{R}}{3.25pt}} \bY(t)$ $\forall t\geq 0$, where $\bY(t)$ is the inverse Laplace transform of $y_{\scaleto{{R}}{3.25pt}}(s)$ in \eqref{eq:Cncontr}. It can be noticed that, for real systems, i.e., systems where $\bM(t)$ is a real function, the matrix $\bV_{\scaleto{{R}}{3.25pt}}$ can be chosen as a real matrix. Also, in the context of time-delay systems, \Cref{def:Cncontr}  is  equivalent to the one of  point-wise complete controllability, see \cite{weiss1967controllability,yi2008controllability}  and of $\R^n$ controllability~\cite[Chapter 2]{kamen1978}. Similarly, we use the following concept of observability.

\begin{definition}\label{defLCnObserv} The $\C^n$--observable subspace $\cW_{\scaleto{{O}}{3.25pt}}$ associated to the pair of functions $(\cC(s), \cK(s))$  is the smallest subspace of $\C^n$ which contains $\cK(s)^{-T}\cC(s)^T$,  for all $s \in i\R$. In other words,  if  $\cW_{\scaleto{{O}}{3.25pt}} =\range{\bW_{\scaleto{{O}}{3.25pt}}}$ , where $\bW_{\scaleto{{O}}{3.25pt}}\in \C^{n \times r_{\scaleto{{O}}{3.25pt}}}$ is a full rank matrix and $r_{\scaleto{{O}}{3.25pt}}\leq n$, we have
	\[\cC(s)\cK(s)^{-1} =y_{\scaleto{{O}}{3.25pt}}(s) \bW_{\scaleto{{O}}{3.25pt}}^T , \]
in which $y_{\scaleto{{O}}{3.25pt}}(s)\in \C^{p \times r_{\scaleto{{O}}{3.25pt}} }$ is a matrix of meromorphic functions. We say that $(\cC(s), \cK(s))$  is $\C^n$--observable if $\cW_{\scaleto{{O}}{3.25pt}} = \C^n$.
\end{definition}
The $\C^n$--unobservable subspace consists of the states $\bq_{uo} \in \C^n$  such that \[\cC(s)\cK(s)^{-1}\bq_{uo} =0 \quad \forall s \in i\R. \]
As a consequence, it is characterized by $\cW_{\scaleto{{O}}{3.25pt}}^{\perp}$. Similarly, in the case of first-order systems \eqref{eq:FirstOrderTF}, $\cK(s)^{-T}\cC^T(s) = (s\bI-\bA)^{-T}\bC^T$ and the above definition is  a natural extension of \eqref{eq:CnContrFirstOrder} for observability.  As for the $\C^n$ reachability space, $\bW_{\scaleto{{O}}{3.25pt}}$ can be chosen to be real if the original system represents a real dynamics.

Analogous to first-order systems \eqref{eq:FirstOrderTF}, if a structured system \eqref{eq:gen_trans_nonpar} is not $\C^n$--reachable or $\C^n$--observable, there exists a lower-order system that has the same transfer function. We state the result rigorously  in the following theorem.
\begin{theorem}\label{prop:ControlObservRemov} Let $(\cC(s), \cK(s), \cB(s))$ be a linear structured system of order $n$ as shown in \eqref{eq:gen_trans_nonpar}. If either $(\cC(s),\cK(s))$ is not $\C^n$--observable or $(\cK(s), \cB(s))$ is not $\C^n$--reachable, then there exists a lower-order structured realization $(\chC(s), \chK(s), \chB(s))$ of order $r<n$, realizing the original transfer function, i.e.,
	\[ \cC(s)\cK(s)^{-1}\cB(s) = \chC(s)\chK(s)^{-1}\chB(s), \quad \forall s \in i\R. \]
\end{theorem}
\begin{proof} Let us first consider the case where the structured system is not $\C^n$--reachable. In this case, there exists a full rank matrix $\bV_{\scaleto{{R}}{3.25pt}} \in \C^{n \times r_{\scaleto{{R}}{3.25pt}}}$ such that $\cK(s)^{-1}\cB(s) = \bV_{\scaleto{{R}}{3.25pt}} z(s)$. Consider a matrix $\bW\in \R^{n \times r_{\scaleto{{R}}{3.25pt}}}$ and assume that $\bW^T\cK(s)\bV_{\scaleto{{R}}{3.25pt}} $ is non-singular. Hence, $(\bW^T\cK(s)^{-1}\bV_{\scaleto{{R}}{3.25pt}})^{-1}\bW^T\cB(s) = z(s)$ and we have 
	\begin{equation}\label{eq:1} 
	\cK(s)^{-1}\cB(s) = \bV_{\scaleto{{R}}{3.25pt}}(\bW^T\cK(s)^{-1}\bV_{\scaleto{{R}}{3.25pt}})^{-1}\bW^T\cB(s).  
	\end{equation}
	Thus,
	\[ \cC(s)\cK(s)^{-1}\cB(s) =   \hat{\cC}(s)\hat{\cK}(s)^{-1}\hat{\cB}(s) \]
	with 
	\[ \hat{\cC}(s) = \cC(s)\bV_{\scaleto{{R}}{3.25pt}} \in \C^{p \times r}, ~\,\hat{\cK}(s) = \bW^T\cK(s)\bV_{\scaleto{{R}}{3.25pt}} \in \C^{r \times r},~\,\text{and}~\, \hat{\cB}(s) = \bW^T\cB(s) \in \C^{r \times m}. \]
	Similarly, when the system is not $\C^n$--observable, it can be shown that there exists a lower-order realization. 
\end{proof}
The above result states that if a structured realization is not $\C^n$--reachable or $\C^n$--observable, then there exists a lower-order structured realization which represents the same transfer function.  Moreover, a lower-order realization can be obtained via Petrov-Galerkin projections.  Although the counterpart of \Cref{prop:ControlObservRemov} is valid for first-order systems as \eqref{eq:FirstOrderTF}, until now we have not been able to give a concrete proof for this result for structured systems. Hence, we leave the following statement as a conjecture.
\begin{conject}\label{conj:Minimality} Let $(\cC(s), \cK(s), \cB(s))$ be a linear structured system of order $n$ as given in \eqref{eq:gen_trans_nonpar}. If $(\cC(s), \cK(s))$ is $\C^n$--observable and $(\cK(s), \cB(s))$ is $\C^n$--reachable, then there is no lower-order realization, which can be constructed via Petrov-Galerkin projection that realizes the same transfer function.
\end{conject}

\begin{remark}
For first-order systems \eqref{eq:FirstOrderTF}, \Cref{conj:Minimality} is equivalent to the statement that a system is minimal if and only if it is reachable and observable. Its proof is based on the Markov parameters of the system, which are represented by $\bC\bA^{k}\bB \in \C^{p \times m}$ for $k\geq 0$, and the associated Hankel matrix, see, e.g., \cite[Theorem 2.28]{dullerud2013course}. However, in the structured parametric case, the classical Markov parameters $\bM_k \in \C^{p \times m}$, where $\bH(s) = \sum_{k=0}^{\infty} \bM_k/s^k$, give information about the first-order realization as \eqref{eq:FirstOrderTF} which realizes the transfer function $\bH(s)$. Hence, they are incompetent to the understanding of the structured minimal realization. Additionally,  the Markov parameters considered in the algebraic system theory depend on $s$ and $\bp$ and are  no longer elements of  $\C^{p \times m}$, see \cite{kamen1978}. Hence, a similar argument as for first-order systems does not hold here. For the moment, we let this be a potential research problem, and we will dedicate future work in the direction of finding a good framework where classical linear systems results can be extended to the structured parametric case. 
\end{remark}

We have shown that if we know reachable and observable subspaces $\cV_{\scaleto{{R}}{3.25pt}}$ and $\cW_{\scaleto{{O}}{3.25pt}}$, we are able to construct lower-order systems by removing the states that are unreachable or unobservable. The next step is to characterize these spaces using matrices $\bV$ and $\bW$ defined in \cref{theo:InterpParStruct}.  By definition, it is easy to observe that $\range{\bV} \subseteq  \cV_{\scaleto{{R}}{3.25pt}}$ and  $\range{\bW} \subseteq  \cW_{\scaleto{{O}}{3.25pt}}$. Additionally, there always exist $N \in \mathbb{N}$ and  interpolation points  $\sigma_1, \dots, \sigma_{N}$ such that  
\begin{subequations}\label{eq:VWnonpar}
\begin{align}
\bV &= \begin{bmatrix}
\cK(\sigma_1)^{-1}\cB(\sigma_1) & \cK(\sigma_2)^{-1}\cB(\sigma_2) & \dots & \cK(\sigma_N)^{-1}\cB(\sigma_N)\end{bmatrix}, \\
\bW  &= \begin{bmatrix}
\cK(\sigma_1)^{-T}\cC(\sigma_1)^T & \cK(\sigma_2)^{-T}\cC(\sigma_2)^T & \dots & \cK(\sigma_N)^{-T}\cC(\sigma_N)^T 
\end{bmatrix},  
\end{align}
\end{subequations}
and
\[\range{\bV} = \cV_{\scaleto{{R}}{3.25pt}} \quad \text{and} \quad  \range{\bW} = \cW_{\scaleto{{O}}{3.25pt}}. \]

In what follows, we show that $N$ can be fixed in some cases where the following assumption holds.
\begin{assump}[\bf at most $n_z$ zeros on the imaginary axis] \label{assumpt:FinZero} Given   $n_z \in \N$ and the functions $\cK(s)$ and $\cB(s)$, we say that the pair of functions  $(\cK(s),\cB(s))$ has at most $n_z$  zeros on the imaginary axis, with $n_z\in \N$, if, for every $\bw \in \C^n$  and $\bz \in \C^m$, the scalar meromorphic function
\[ \bF_{\bw,\bz}(s)  = \bw^T\cK(s)^{-1}\cB(s)\bz \]
has at most $n_z$ zeros on the imaginary axis.
\end{assump}

Hence, the following  theorem is an extension of the result for first-order systems from \cite[Lemma 3.1]{morAndA90} to linear structured systems.
\begin{theorem}\label{theo:FinNunZero} Let $(\cK(s),\cB(s))$ be a pair of matrix functions whose reachable subspace is $\cV_{\scaleto{{R}}{3.25pt}}$. Let us assume that the pair $(\cK(s),\cB(s))$ has at most $n_z$ zeros on the imaginary axis. Let
\[ \bV = \begin{bmatrix}
\cK(\sigma_1)^{-1}\cB(\sigma_1) & \cK(\sigma_2)^{-1}\cB(\sigma_2) & \dots & \cK(\sigma_N)^{-1}\cB(\sigma_N)\end{bmatrix},
\]
with $N\geq n_z$. Hence, 
\[\cV_{\scaleto{{R}}{3.25pt}} =  \range{\bV_{\scaleto{{R}}{3.25pt}}} = \range{\bV}.  \]
\end{theorem}
\begin{proof}
Suppose  that $\cK(s)^{-1}\cB(s)$ has at most $n_z$ zeros on $i\R$. We know that $ \range{\bV} \subseteq \cV_{\scaleto{{R}}{3.25pt}}$. Suppose by contradiction that  $\range{\bV} \neq \cV_{\scaleto{{R}}{3.25pt}}$. Hence, there exists a vector $\bw \in \cV_{\scaleto{{R}}{3.25pt}}$ such that $\bw^T\bV = 0$. Hence,
$\bw^T \cK(\sigma_j)^{-1}\cB(\sigma_j)e_i =0$, for $j =1, \dots, N$, $i =1, \dots, m$. Hence, the function $\bw^T \cK(s)^{-1}\cB(s)e_i =0$ has at least $N$ zeros on the imaginary axis, with  $N\geq n_z$, which contradicts the assumption of $n_z$ zeros.
\end{proof}
\Cref{theo:FinNunZero} shows that, under the hypothesis of a limited number of zeros on the imaginary axis, the reachability subspace is encoded by the matrix $\bV$ constructed with $N$ interpolation points, where $N \geq n_z$. 

\begin{remark} In order to have $n_z$ from  \Cref{theo:FinNunZero}, \Cref{assumpt:FinZero} needs to be fulfilled. However,  this assumption is difficult to verify in practice. Nevertheless, numerical experiments show that it is enough to choose $N = n$, so that $\range{\bV} = \cV_{\scaleto{{R}}{3.25pt}}$.  
Moreover, in a large-scale setting $(n \in \cO(10^{5}{-}10^{6}))$, we observe that as we increase the number of interpolation points $(N)$, $\range{\bV} \to \cV_{\scaleto{{R}}{3.25pt}}$.
\end{remark}
The same analysis presented here holds for the observability space. 
\subsection{Note on the Parametric Case} 
So far, we have presented the definitions of $\C^n$ reachability and $\C^n$ observability for structured non-parametric systems. For the parametric case  \eqref{eq:gen_trans}, we can consider the following natural extensions of \Cref{def:Cncontr,defLCnObserv}. The reachable subspace $\cV_{\scaleto{{R}}{3.25pt}} = \range{\bV_{\scaleto{{R}}{3.25pt}}}$ associated to the pair of functions  $(\cK(s,\bp), \cB(s,\bp))$, with a full column rank matrix $\bV_{\scaleto{{R}}{3.25pt}}\in \C^{n \times r_{\scaleto{{R}}{3.25pt}}}$, $r_{\scaleto{{R}}{3.25pt}} \leq n$, is the smallest subspace of $\C^n$ which contains $\cK(s,\bp)^{-1}\cB(s,\bp)$ for every $s \in i\R$ and $\bp \in \Omega^d$. In other words,
\[\cK(s, \bp)^{-1}\cB(s,\bp) = \bV_{\scaleto{{R}}{3.25pt}} \bz_c(s,\bp), \]
where $\bz_c(s, \bp)\in \C^{p \times r_{\scaleto{{R}}{3.25pt}}}$.  Similarly, the $\C^n$--observable subspace $\cW_{\scaleto{{O}}{3.25pt}} = \range{\bW_{\scaleto{{O}}{3.25pt}}}$ associated to $(\cC(s, \bp), \cK(s, \bp))$, with a full rank matrix $\bW_{\scaleto{{O}}{3.25pt}}\in \C^{n \times r_{\scaleto{{O}}{3.25pt}}}$, $r_{\scaleto{{O}}{3.25pt}}\leq n$, is the smallest subspace of $\C^n$ which contains $\cK(s,\bp)^{-T}\cC(s,\bp)^T$ for every $s \in i\R$ and $\bp \in \Omega^d$. In other words,  
\[\cC(s,\bp)\cK(s, \bp)^{-1} =y_{\scaleto{{O}}{3.25pt}}(s, \bp) \bW_{}^T \bW_{\scaleto{{O}}{3.25pt}} , \]
where $y_{\scaleto{{O}}{3.25pt}}(s, \bp)\in \C^{p \times r_{\scaleto{{O}}{3.25pt}} }$ is a matrix of functions. 
Furthermore, the result of \Cref{prop:ControlObservRemov}  can be extended to the parametric case, showing that states that are unreachable and unobservable can be removed.   We do not have an analogue for \Cref{theo:FinNunZero}. However, we know that if we take enough interpolations points $(\sigma_i,\bp_i)$, the matrices
\begin{subequations}\label{eq:VWpar}
\begin{align}
\bV &= \begin{bmatrix}
\cK(\sigma_1,\bp_1)^{-1}\cB(\sigma_1,\bp_1) & \dots & \cK(\sigma_{N}, \bp_{N})^{-1}\cB(\sigma_{N},\bp_{N)}
\end{bmatrix}, \\
\bW  &= \begin{bmatrix}
\cK(\sigma_1,\bp_1)^{-T}\cC(\sigma_1,\bp_1)^T &  \dots & \cK(\sigma_N,\bp_{N})^{-T}\cC(\sigma_{N}, \bp_{N})^T 
\end{bmatrix},  
\end{align}
\end{subequations}
encode the reachability and observability spaces, i.e., \[\range{\bV} = \cV_{\scaleto{{R}}{3.25pt}} \quad \text{and} \quad \range{\bW} = \cW_{\scaleto{{O}}{3.25pt}}.\] 

From now on, we assume that we have  enough interpolation points such that $\range{\bV} = \cV_{\scaleto{{R}}{3.25pt}}$ and $\range{\bW} = \cW_{\scaleto{{O}}{3.25pt}}$.

\subsection{Simultaneous Reduction}
We have seen that if the system \eqref{eq:gen_trans} is not reachable or not observable, there exists a lower-order realization whose transfer function remains the same as the original one. To obtain such a lower-order realization, one needs to truncate the states that are unreachable and observable, as shown in \Cref{prop:ControlObservRemov}. However, it remains an open question of how this information can be used to obtain reduced-order systems. In what follows, we propose a method enabling to identify simultaneously the states that are unreachable and unobservable. For this, we assume that  
\begin{equation}\label{eq:RankCondition}
\rank{\begin{bmatrix}
	\bW^T\bA_1\bV & \dots & \bW^T\bA_l\bV
	\end{bmatrix}} = \rank{\begin{bmatrix}
	\bW^T\bA_1\bV \\ \vdots \\ \bW^T\bA_l\bV
	\end{bmatrix}} = r,
\end{equation}
where $\bV$ and $\bW$ are the matrices defined in \cref{eq:VWnonpar} encoding, respectively, the reachability and the observability subspaces, i.e.,    $\range{\bV} = \cV_{\scaleto{{R}}{3.25pt}}$ and $\range{\bW} = \cW_{\scaleto{{O}}{3.25pt}}$ and $\bA_i, i \in \{1,\ldots,l\}$, are defined in \eqref{eq:sys_matrices}. Then, we consider the compact SVDs
\begin{equation}\label{eq:shortSVD}
\begin{bmatrix}
\bW^T\bA_1\bV & \dots & \bW^T\bA_l\bV
\end{bmatrix} = \bW_1 \Sigma_l \tilde{\bV}^T \quad \text{and} \quad  \begin{bmatrix}
\bW^T\bA_1\bV \\ \vdots \\ \bW^T\bA_l\bV
\end{bmatrix} = \tilde{\bW}\Sigma_r \bV_1^T.
\end{equation}
Let $\bW_p := \bW\bW_1$ and  $\bV_p := \bV\bV_1$ be two projection matrices and let us consider the lower-order realization $\chC_p(s, \bp) \chK_p(s, \bp)^{-1}\chB_p(s, \bp)$ constructed by Petrov-Galerkin projection as follows:
\begin{equation}\label{eq:MinimalReal} \chB_p(s,\bp) = \bW_p^T \cB(s, \bp), \quad \chC_p(s, \bp) = \cC(s, \bp)\bV_p, \quad \chK_p(s,\bp) = \bW_p^T\cK(s, \bp)\bV_p.    
\end{equation}
Then, the following result holds.
\begin{theorem}\label{theo:MinReal} The lower-order system $\chC_p(s,\bp) \chK_p(s,\bp)^{-1}\chB_p(s, \bp)$ of order $r$, obtained as given in \cref{eq:MinimalReal}, realizes the original transfer function, i.e.,
	\[\chC_p(s, \bp) \chK_p(s, \bp)^{-1}\chB_p(s, \bp) = \cC(s, \bp) \cK(s,\bp)^{-1}\cB(s,\bp) \]
for every $s\in i\R$ and $\bp \in \Omega^d$.
\end{theorem}
\begin{proof}
	Recall that $\bV_p = \bV\bV_1$ and $\bW_p = \bW\bW_1$. Hence, by construction, $\range{\bV_p} \subseteq \range{\bV} = \cV_{\scaleto{{R}}{3.25pt}}$ and $\range{\bW_p}  \subseteq \range{\bW} =  \cW_{\scaleto{{O}}{3.25pt}}$.  Hence, there exist $\bV_a$ and $\bW_a$, $\bV_a \perp \bV_p$ and $\bW_a \perp \bW_p$  such that $\range{\begin{bmatrix}\bV_p & \bV_a\end{bmatrix}} = \cV_{\scaleto{{R}}{3.25pt}}$ and $\range{\begin{bmatrix}\bW_p & \bW_a \end{bmatrix}} = \cW_{\scaleto{{O}}{3.25pt}}$. 
	As  a consequence, from the compact SVDs in \eqref{eq:shortSVD},
	\[ \bW^T\bA_k\bV_a = 0 ,\quad \bW_a^T\bA_k\bV = 0, \quad \text{and} \quad \bW_a^T\bA_k\bV_a =0, ~\text{for $k = 1, \dots ,l$},   \]
	and hence,
	\[\bW_p^T\cK(s, \bp)\bV_a = 0, \quad \bW_a^T\cK(s, \bp)\bV_p = 0, \quad \text{and}~\, \quad \bW_a^T\cK(s, \bp)\bV_a = 0. \]
	Moreover, notice that
	\begin{align*}
	&\cK(s,\bp)^{-1}\cB(s, \bp) = \begin{bmatrix}
	\bV_p  & \bV_a
	\end{bmatrix}\begin{bmatrix}
	\chK_p(s, \bp)^{-1}\chB_p(s, \bp) \\
	\star
	\end{bmatrix}\quad \text{and}\\ 
	&\cC(s, \bp)\cK(s, \bp)^{-1} = \begin{bmatrix}
	\chC_p(s, \bp)\chK_p(s, \bp)^{-1} & \star
	\end{bmatrix}\begin{bmatrix}
	\bW_p^T \\ 
	\bW_a^T
	\end{bmatrix}.  
		\end{align*}
	Finally, we write
	\begin{align*}
	\cC(s, \bp)&\cK(s, \bp)^{-1}\cB(s, \bp) =  \cC(s, \bp)\cK(s)^{-1}\cK(s,\bp)\cK(s,\bp)^{-1}\cB(s,\bp) 
	\\
	&= \begin{bmatrix}
	\chC(s, \bp)\chK(s, \bp)^{-1} & \star
	\end{bmatrix}\begin{bmatrix}
	\bW_p^T \\ 
	\bW_a^T
	\end{bmatrix}\cK(s, \bp)\begin{bmatrix}
	\bV_p  & \bV_a
	\end{bmatrix}\begin{bmatrix}
	\chK(s, \bp)^{-1}\chB(s, \bp) \\
	\star
	\end{bmatrix} 
	\\ 
	&= \begin{bmatrix}
	\chC(s, \bp)\chK(s, \bp)^{-1} & \star
	\end{bmatrix}\begin{bmatrix}
	\bW_p^T\cK(s, \bp)\bV_p & 0 \\
	0 & 0
	\end{bmatrix}\begin{bmatrix}
	\chK(s, \bp)^{-1}\chB(s, \bp) \\
	\star
	\end{bmatrix} 
	\\
	&= 
	\chC(s, \bp)\chK(s, \bp)^{-1} \chK(s, \bp) 
	\chK(s, \bp)^{-1}\chB(s, \bp)  =  \chC(s, \bp)\chK(s, \bp)^{-1} \chB(s, \bp), 
	\end{align*}
which concludes the proof.
\end{proof}
\Cref{theo:MinReal} shows that if we construct the lower-order system as shown in~\cref{eq:MinimalReal}, it also realizes the original system. As a consequence, the rank condition \cref{eq:RankCondition} gives us the order $r$ of the lower-order realization constructed by Petrov Galerkin projection which realizes the original system. Hence, the matrices 
\begin{equation}\label{eq:ComplexMatrices}
\begin{bmatrix}
	\bW^T\bA_1\bV & \dots & \bW^T\bA_l\bV
\end{bmatrix}, \quad \text{and}~\, \begin{bmatrix}
	\bW^T\bA_1\bV \\ \vdots \\ \bW^T\bA_l\bV
\end{bmatrix}
\end{equation}
encode the complexity of the original system. Moreover, by means of the compact SVDs of these matrices, we are able to find the projection matrices $\bV_p = \bV\bV_1$ and $\bW_p = \bW\bW_1$, leading to a lower-order system having the same transfer function as the original one.  
Hence, the states that are removed from the dynamics of the original systems can be seen as unreachable or unobservable ones.  On the other hand,   the singular values in $\Sigma_r$ and $\Sigma_l$ give some important information about the simultaneous degree of reachability and observability of the states. Indeed, states that are related to small singular values can be interpreted to have a weak simultaneous degree of reachability  and observability, while the states related to large singular values are strongly simultaneously reachable and observable. Therefore, removing also the subspaces associated  with small singular values leads to reduced-order systems. 

\section{Model-Order Reduction Algorithm}
Based on the arguments given in the previous section, we propose \Cref{algo:Inter_structured_para} enabling us to construct reduced-order systems for structured systems \eqref{eq:gen_trans_nonpar}.  The procedure consists in selecting interpolation points $(s_i, \bp_i)$ and constructs the matrices $\bV$ and $\bW$ as in \eqref{eq:ProjMatrices} (steps 2 and 3). If we have enough interpolation points, the subspaces $\range{\bV}$ and $\range{\bW}$  mimic the reachable subspace $\cV_{\scaleto{{R}}{3.25pt}}$ and the observable  $\cW_{\scaleto{{O}}{3.25pt}}$, respectively. Then, in step 4, we compute the SVDs of the matrices in \eqref{eq:SVDalgo}. As previously discussed, the numerical rank of those matrices  gives the order of a good reduced-order system. Hence, in step 5, the projection matrices $\bV_p$ and $\bW_p$ are constructed according to \Cref{theo:MinReal}. Finally, in step 6, the reduced-order system  is determined by Petrov-Galerkin projection.

	\begin{remark}\label{rem:tangentialinterpolation}
		So far in the paper, we have refrained from discussing the idea of tangential interpolation. In the case of multi-input multi-output (MIMO) systems, one can employ the idea of tangential interpolation which has been proven to be very useful in MIMO case. For this, along with considering interpolation points $\{\sigma_i,\bp^r_i\}$ and $\{\mu_j,\bp^l_j\}$, we also consider tangential directions $\bb_i$ and $\bc_j$ of appropriate sizes. For details, we refer to e.g., \cite{morBauBBetal11}. Hence, with the tangential directions, the analogues of matrices $\bV$ and $\bW$ in \eqref{eq:VWpar} are 
		\begin{align*}
		\bV &= \begin{bmatrix}
		\cK(\sigma_1,\bp_1)^{-1}\cB(\sigma_1,\bp_1)\bb_1 & \dots & \cK(\sigma_{N}, \bp_{N})^{-1}\cB(\sigma_{N},\bp_{N)}\bb_{N}
		\end{bmatrix}, \\
		\bW  &= \begin{bmatrix}
		\cK(\sigma_1,\bp_1)^{-T}\cC(\sigma_1,\bp_1)^T\bc_1^T &  \dots & \cK(\sigma_N,\bp_{N})^{-T}\cC(\sigma_{N}, \bp_{N})^T\bc_{N}^T 
		\end{bmatrix}.
		\end{align*} 
\end{remark}
Hence, these matrices can be used in step 3 of  \Cref{algo:Inter_structured_para}. 
\begin{remark}
 One of the additional advantages of the proposed algorithm is that it inherently allows us to construct frequency-limited reachable and observable subspaces by choosing the interpolation points in a given frequency range. Hence, it yields a reduced-order system that is good in the considered frequency range. 
\end{remark}
\begin{remark}
 The proposed framework is also suitable for non-dynamical linear parametric systems, i.e., systems of the form
\begin{align*}
\bA(\bp)\bx(\bp) &= \bB(\bp), \\
\by(\bp) &= \bC(\bp)\bx(\bp).
\end{align*}
In this case, the solution $\by(\bp)$ can be obtained as
\[
\by(\bp) = \bC(\bp) \bA(\bp)^{-1} \bB(\bp),
\]
and the reachability and observability subspaces, namely $\cV_{\scaleto{{R}}{3.25pt}}$ and $\cW_{\scaleto{{O}}{3.25pt}}$, are the smallest subspaces such that $\bA(\bp)^{-1} \bB(\bp)\in\cV_{\scaleto{{R}}{3.25pt}}$  and $ \bA(\bp)^{-T}\bC(\bp)^T\in \cW_{\scaleto{{O}}{3.25pt}}$ for all $\bp \in \Omega^d$. In this case, we determine dominant subsystems with respect to parameters using the \Cref{algo:Inter_structured_para}.
\end{remark}
\begin{remark}
 In several applications, it is highly desirable to determine reduced-order systems via one-sided projection in order to potentially preserve some important properties of the systems such as stability or passivity. In this case, in step 3 of \Cref{algo:Inter_structured_para}, we set $\bW = \bV$, and in step 5 of \Cref{algo:Inter_structured_para}, we again set $\bW_p = \bV_p$. 
\end{remark}

\begin{algorithm}[tb]
	\caption{Construction of ROMs via \textbf{D}ominant \textbf{R}eachable and \textbf{O}bservable subspace-based \textbf{P}rojection (\textbf{DROP}).}\label{algo:Inter_structured_para}
	\begin{algorithmic}[1]
				\State \textbf{Input:} The transfer function as in \eqref{eq:gen_parametric_TF}, and order of reduced system $r$.
				\State Choose left and right interpolation points for frequency and parameters. 
				\State Compute $\bV$ and $\bW$ using the interpolation points as in \eqref{eq:ProjMatrices}.
				\State Determine SVDs  \begin{equation}\label{eq:SVDalgo}
				\begin{bmatrix}
				\bW^T\bA_1\bV & \dots & \bW^T\bA_l\bV
				\end{bmatrix} = \bW_1 \Sigma_l \tilde{\bV}^T \quad \text{and} \quad  \begin{bmatrix}
				\bW^T\bA_1\bV \\ \vdots \\ \bW^T\bA_l\bV
				\end{bmatrix} = \tilde{\bW}\Sigma_r \bV_1^T 
				\end{equation} \label{step:SVD}
				\State Compute projection matrices: 
				 $\bV_p = \bV\bV_{1}(:,1:r)$ and $\bW_p = \bW\bW_{1}(:,1:r) $.
				\State Compute reduced matrices:
				\Statex \hfill $\hat{\cK}(s,\bp) = \bW_p^T\cK_s(s)\bV_p$, $\hat{\cB}(s,\bp) = \bW_p^T\cB(s,\bp)$ and $\hat{\cC}(s,\bp) = \cC(s,\bp)\bV_p$.
				\State \textbf{Output:} The reduced-order matrices: $\hat{\cK}(s,\bp),\hat{\cB}(s,\bp),$ and $\hat{\cC}(s,\bp)$.
	\end{algorithmic}
\end{algorithm}

\section{Numerical Results}\label{sec:numerical}
In this section, we illustrate the efficiency of the proposed methods via several numerical examples, arising in various applications. We also compare for the non-parametric case with the existing 
methods   proposed in \cite{breiten2016structure, morFenB10}, and for the parametric case with the adaptive interpolation method proposed in \cite{morFenAB17}. We have performed all the simulations on  a board with 4 \intel ~\xeon~E7-8837 CPUs with a 2.67-GHz clock speed using \matlab ~8.0.0.783 (R2016b). Furthermore, we generate random numbers, whenever necessary, using \texttt{rng(0,`twister')}. In the case of a MIMO system, we determine the projection matrices by employing the idea of tangential interpolation, and we choose the tangential directions randomly. 
\subsection{A Demo Example}
At first, we discuss an artificial example to illustrate the proposed methods. Let us consider a system of order $n=3$ as follows:
\begin{equation}\label{eq:demo_exa}
\begin{aligned}
 \bbm 1 & 0 & 0 \\0& 1 & 0 \\ 0&0&1\ebm \dot{\bx}(t)& = \bbm -2 & 0 &0 \\ 0 & -1& 0 \\ 0 & 0 &-2\ebm \bx(t) + \bp\bbm 0 & 1&0 \\ -1 & 0& 0 \\ 1 & 0 & 0 \ebm \bx(t) + \bbm 1\\0\\1 \ebm \bu(t),\\
 \by(t) &= \bbm 1&1&0\ebm \bx(t).
 \end{aligned}
 \end{equation}
The example is constructed in such a way that the parameter $\bp$ changes the imaginary parts of two eigenvalues of the system. Hence, with a change of the parameter, we expect a change in the peak of the transfer function. 

Next, we aim at constructing a reduced realization of the system~\cref{eq:demo_exa} by employing \cref{algo:Inter_structured_para}. 
For this, we take $10$ points for the frequency $s \in \bbm 10^{-4}, 10\ebm$ and for the parameter $\bp \in \bbm -10,10\ebm$. We take the frequency points in the given range in the logarithm scale, whereas the parameters are chosen randomly in the considered parameter range. In \Cref{fig:demo_SVD}, we plot the decay of the singular values as shown in step $4$ of \Cref{algo:Inter_structured_para}. It can be observed that after the first two singular values, the singular values are at the level of machine precision. This means that a reachable and observable system, representing the input/output behavior of the system \eqref{eq:demo_exa} has the order exactly $r=2$. We compare the transfer functions of the original and reduced-order systems  for $20$ linearly spaced parameter values in $\bbm -10,10\ebm$, which are plotted in \Cref{fig:demo_exam_TF}. The figure shows that the error between the original and reduced-order system is of the  level of machine  precision.
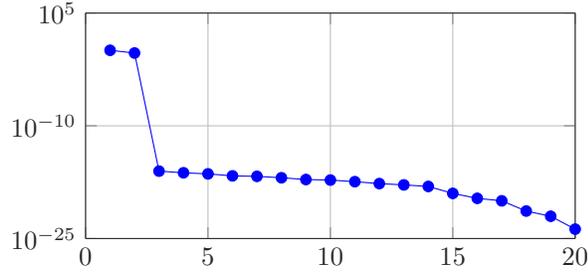
\begin{figure}[!tb]
\centering
\setlength\fheight{3.0cm}
\setlength\fwidth{.5\textwidth}
    \tikzsetnextfilename{Figures/Random_Singular}%
%
%
\begin{tikzpicture}

\begin{axis}[%
width=\fwidth,
height=\fheight,
scale only axis,
separate axis lines,
every outer x axis line/.append style={white!15!black},
every x tick label/.append style={font=\color{white!15!black}},
xmin=0,
xmax=20,
xmajorgrids,
every outer y axis line/.append style={white!15!black},
every y tick label/.append style={font=\color{white!15!black}},
ymode=log,
ymin=1e-25,
ymax=100000,
yminorticks=true,
ymajorgrids,
yminorgrids,
title style={font=\bfseries},
]
\addplot [color=blue,mark=*,mark options={solid},forget plot, mark repeat = 1]
  table[row sep=crcr]{1	1.10456104775585\\
2	0.479948551427472\\
3	9.28118132740165e-17\\
4	5.61935661762386e-17\\
5	4.02271515436719e-17\\
6	2.18211940715568e-17\\
7	1.8798081019259e-17\\
8	1.21917865350396e-17\\
9	7.08588523047599e-18\\
10	6.01774965702497e-18\\
11	3.66906486263913e-18\\
12	2.04476691577805e-18\\
13	1.36614147168194e-18\\
14	8.49695472387931e-19\\
15	1.01758732767455e-19\\
16	2.22889760833644e-20\\
17	1.04191802709932e-20\\
18	4.54976210130993e-22\\
19	9.23444176315289e-23\\
20	1.77427965494288e-24\\
};
\end{axis}
\end{tikzpicture}%
%

\caption{Demo example: Decay of the singular values.}
\label{fig:demo_SVD}
\end{figure}

The order 2 of a minimal realization can also be easily verified by analyzing the system \eqref{eq:demo_exa}. Notice that $x_3(t)$ does not influence the dynamics of $x_1(t)$  and $x_2(t)$, where $x_i(t)$ denotes the $i$th component of the vector $x(t)$, and $x_3(t)$ is not being observed since the output matrix $\bC = \bbm 1,1,0\ebm$. Hence, $x_3(t)$ can be eliminated as far as the input-output dynamics are concerned. Therefore, the system \eqref{eq:demo_exa} can exactly be reduced to order $2$. 

\begin{figure}[!tb]
    \centering
    \setlength\fheight{2.5cm}
    \setlength\fwidth{.7\textwidth}
    \tikzsetnextfilename{Figures/Random_TF_AbsError}%
    \input{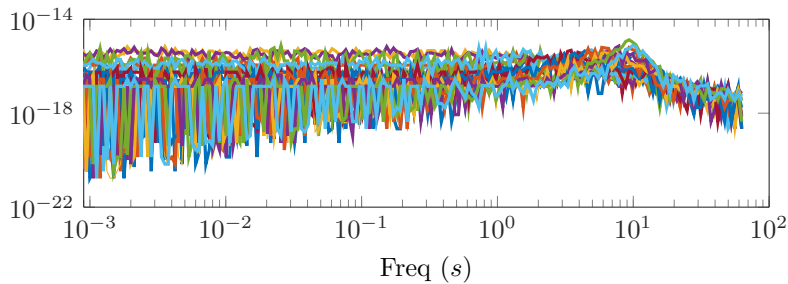}%

    \caption{Demo example: The figure shows the Bode plot of the error between the original and reduced-order systems for different parameter values.}
    \label{fig:demo_exam_TF}
\end{figure}

\subsection{Time-Delay System}
Next, we consider the time-delay model \cite{morBeaG09} of order $n$ as  follows:
\begin{equation}
 \begin{aligned}
  \bE\dot{\bx}(t) & =  \bA \bx(t) + \bA_\tau \bx(t-\tau) + \bB \bu(t),\\
 \by(t) & = \bC \bx(t),
 \end{aligned}
\end{equation}
where $\bE = \mu \bI_n + \bT $, $\bA = \tfrac{1}{\tau}\left(\tfrac{1}{\zeta} + 1\right)\left(\bT - \mu \bI_n\right)$, and $\bA_\tau = \tfrac{1}{\tau}\left(\tfrac{1}{\zeta} - 1\right)\left(\bT - \mu \bI_n\right)$ in which 
the matrix $\bT$ is such that it is ones on the sub- and super-diagonal along with the $(1,1)$ and $(n,n)$ elements. The input matrix $\bB$ is zero everywhere except for the first and second entries, i.e., $\bB(1) = \bB(2) = 1$, and the output matrix is $\bC = \bB^T$. Furthermore, we choose $\mu = 5$, $\zeta = 0.01$, $\tau = 1$, and the order $n = 500$.

Next, we aim at constructing reduced-order systems using  balanced truncation as proposed in \cite{breiten2016structure} and using \Cref{algo:Inter_structured_para}. In order to apply \Cref{algo:Inter_structured_para}, we take $1000$ logarithmically distributed frequency points in the range $\bbm 10^{-2},10^4\ebm$. Furthermore, to apply balanced truncation, we need to determine the system Gramians, which are given in integral forms. 
To compute approximations of the Gramians, we make of use of \texttt{quadv} command in \matlab~with $\texttt{tol}=10^{-10}$ to integrate in the given frequency range. 

First, we plot the singular values, obtained from balanced truncation and \Cref{algo:Inter_structured_para} in \Cref{fig:delay_SVD}, which indicates a faster decay of singular values obtained from \Cref{algo:Inter_structured_para}.   We now determine reduced-order systems of order $r = 12$ using both methods. We compare the Bode plots of the original and reduced-order systems in \Cref{fig:delay_exam_TF}. The figure shows that both methods capture the dynamics very well; however, our method clearly yields a better reduced-order system  at least by two orders of magnitudes. 

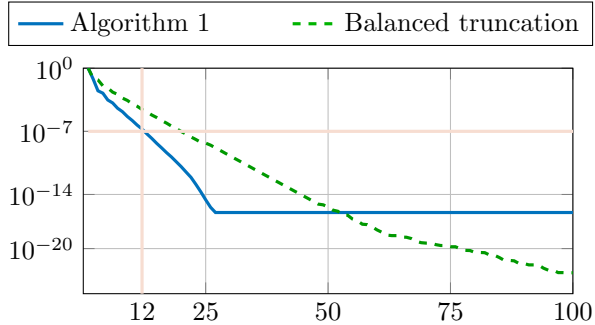
\begin{figure}[!tb]
\centering
\definecolor{mycolor1}{rgb}{0.00000,0.44700,0.74100}%
\definecolor{mycolor2}{rgb}{0.85000,0.32500,0.09800}%
\definecolor{mycolor3}{rgb}{0.92900,0.69400,0.12500}%
\begin{tikzpicture}
    \begin{customlegend}[legend columns=-1, legend style={/tikz/every even column/.append style={column sep=1cm}} , legend entries={Algorithm 1, Balanced truncation }, ]
    \addlegendimage{mycolor1, line width = 1.2pt}
    \addlegendimage{black!30!green,dashed, line width = 1.2pt}
    \end{customlegend}
\end{tikzpicture}
\setlength\fheight{3.0cm}
\setlength\fwidth{.5\textwidth}
    \tikzsetnextfilename{Figures/Delay_DecaySV}%
%
%
\definecolor{mycolor1}{rgb}{0.00000,0.44700,0.74100}%
\definecolor{mycolor2}{rgb}{0.85000,0.32500,0.09800}%
\definecolor{mycolor3}{rgb}{0.92900,0.69400,0.12500}%
\begin{tikzpicture}

\begin{axis}[%
width=1\fwidth,
height=\fheight,
at={(0\fwidth,0\fheight)},
scale only axis,
xmin=0,
xmax=100,
ymode=log,
ymin=1e-25,
ymax=1,
yminorticks=true,
grid = major,
xtick = {12,25,50,75,100}, 
ytick = {1e0,1e-7,1e-14,1e-20}, 
axis background/.style={fill=white},
legend style={legend cell align=left, align=left, draw=white!15!black}
]
\addplot [color=mycolor1, line width = 1.2pt]
  table[row sep=crcr]{%
1	1\\
2	0.0654198180593519\\
3	0.00310810036044955\\
4	0.00169792098469833\\
5	0.000299409386531134\\
6	0.000140149493319843\\
7	3.64431972783089e-05\\
8	1.50598353284312e-05\\
9	4.38516713691885e-06\\
10	1.58531623557241e-06\\
11	4.93226915945008e-07\\
12	1.8149549906185e-07\\
13	5.19646172130738e-08\\
14	1.62787086923027e-08\\
15	4.63770646241635e-09\\
16	1.46148705170406e-09\\
17	4.37453741122567e-10\\
18	1.45210931410279e-10\\
19	4.18579733894498e-11\\
20	1.20464585890012e-11\\
21	3.0426140394736e-12\\
22	7.87725947490282e-13\\
23	1.31146353325341e-13\\
24	1.96263647741073e-14\\
25	2.55977559489795e-15\\
26	3.77907322841843e-16\\
27	9.99194704669846e-17\\
28	9.99194704669846e-17\\
29	9.99194704669846e-17\\
30	9.99194704669846e-17\\
31	9.99194704669846e-17\\
32	9.99194704669846e-17\\
33	9.99194704669846e-17\\
34	9.99194704669846e-17\\
35	9.99194704669846e-17\\
36	9.99194704669846e-17\\
37	9.99194704669846e-17\\
38	9.99194704669846e-17\\
39	9.99194704669846e-17\\
40	9.99194704669846e-17\\
41	9.99194704669846e-17\\
42	9.99194704669846e-17\\
43	9.99194704669846e-17\\
44	9.99194704669846e-17\\
45	9.99194704669846e-17\\
46	9.99194704669846e-17\\
47	9.99194704669846e-17\\
48	9.99194704669846e-17\\
49	9.99194704669846e-17\\
50	9.99194704669846e-17\\
51	9.99194704669846e-17\\
52	9.99194704669846e-17\\
53	9.99194704669846e-17\\
54	9.99194704669846e-17\\
55	9.99194704669846e-17\\
56	9.99194704669846e-17\\
57	9.99194704669846e-17\\
58	9.99194704669846e-17\\
59	9.99194704669846e-17\\
60	9.99194704669846e-17\\
61	9.99194704669846e-17\\
62	9.99194704669846e-17\\
63	9.99194704669846e-17\\
64	9.99194704669846e-17\\
65	9.99194704669846e-17\\
66	9.99194704669846e-17\\
67	9.99194704669846e-17\\
68	9.99194704669846e-17\\
69	9.99194704669846e-17\\
70	9.99194704669846e-17\\
71	9.99194704669846e-17\\
72	9.99194704669846e-17\\
73	9.99194704669846e-17\\
74	9.99194704669846e-17\\
75	9.99194704669846e-17\\
76	9.99194704669846e-17\\
77	9.99194704669846e-17\\
78	9.99194704669846e-17\\
79	9.99194704669846e-17\\
80	9.99194704669846e-17\\
81	9.99194704669846e-17\\
82	9.99194704669846e-17\\
83	9.99194704669846e-17\\
84	9.99194704669846e-17\\
85	9.99194704669846e-17\\
86	9.99194704669846e-17\\
87	9.99194704669846e-17\\
88	9.99194704669846e-17\\
89	9.99194704669846e-17\\
90	9.99194704669846e-17\\
91	9.99194704669846e-17\\
92	9.99194704669846e-17\\
93	9.99194704669846e-17\\
94	9.99194704669846e-17\\
95	9.99194704669846e-17\\
96	9.99194704669846e-17\\
97	9.99194704669846e-17\\
98	9.99194704669846e-17\\
99	9.99194704669846e-17\\
100	9.99194704669846e-17\\
};

\addplot [color=mycolor2]
  table[row sep=crcr]{%
12	1e-20\\
12	1\\
};

\addplot [color=green!60!black, line width=1.2pt, dashed]
  table[row sep=crcr]{%
1	1\\
2	0.0825605822628279\\
3	0.0502859070459224\\
4	0.0110450514381799\\
5	0.00651907362770264\\
6	0.00219982964095309\\
7	0.00125877105004681\\
8	0.000484860199431572\\
9	0.000267560764408225\\
10	0.000112168961722707\\
11	6.06260896851046e-05\\
12	2.65849686510803e-05\\
13	1.41549995162061e-05\\
14	6.40656171212832e-06\\
15	3.37990186561874e-06\\
16	1.561891651034e-06\\
17	8.1864635463507e-07\\
18	3.84242235254525e-07\\
19	2.00499458503916e-07\\
20	9.51997088805554e-08\\
21	4.95262779164457e-08\\
22	2.37375530144125e-08\\
23	1.23879519696425e-08\\
24	6.11951684916243e-09\\
25	4.96094867531736e-09\\
26	2.94033786545689e-09\\
27	1.47167671605053e-09\\
28	7.50499212007337e-10\\
29	3.71778432606846e-10\\
30	1.89487785818331e-10\\
31	9.40698139733285e-11\\
32	4.79442955719793e-11\\
33	2.38564545527545e-11\\
34	1.21593096642302e-11\\
35	6.06239754137469e-12\\
36	3.09008968422665e-12\\
37	1.54338758473217e-12\\
38	7.86767499740298e-13\\
39	3.93627357709321e-13\\
40	2.00691815353486e-13\\
41	1.00640882058388e-13\\
42	5.13188371333668e-14\\
43	2.58691478322983e-14\\
44	1.31790576677334e-14\\
45	6.7094974582247e-15\\
46	3.39920102499444e-15\\
47	1.69515057583872e-15\\
48	8.49892860282763e-16\\
49	7.11464276804845e-16\\
50	4.00183624781162e-16\\
51	1.8290734040803e-16\\
52	1.37091378067881e-16\\
53	4.94796275088994e-17\\
54	4.94796275088994e-17\\
55	1.87523297863654e-17\\
56	1.21111197029022e-17\\
57	5.02322557497073e-18\\
58	4.54309840007012e-18\\
59	2.60017624918136e-18\\
60	1.17720358119794e-18\\
61	6.38831459966269e-19\\
62	3.13266205081378e-19\\
63	2.74393310593922e-19\\
64	2.74393310593922e-19\\
65	1.88816500891419e-19\\
66	1.81053056109519e-19\\
67	1.29590332928363e-19\\
68	6.53002098019709e-20\\
69	4.52190898105276e-20\\
70	4.0631932863784e-20\\
71	4.0631932863784e-20\\
72	2.75705860590856e-20\\
73	2.19408123787932e-20\\
74	2.19408123787932e-20\\
75	1.59914901035726e-20\\
76	1.59914901035726e-20\\
77	9.56853754066448e-21\\
78	6.33632407388479e-21\\
79	6.06118190793237e-21\\
80	3.70436433563131e-21\\
81	3.70436433563131e-21\\
82	3.52911002222037e-21\\
83	1.92539888507961e-21\\
84	1.80768529101389e-21\\
85	1.00364316191861e-21\\
86	4.88438205743009e-22\\
87	4.88438205743009e-22\\
88	4.4116755540395e-22\\
89	2.23202154812294e-22\\
90	2.23202154812294e-22\\
91	1.55977704844988e-22\\
92	1.55977704844988e-22\\
93	1.49552669776656e-22\\
94	6.74204481170499e-23\\
95	3.98731664363674e-23\\
96	3.98731664363674e-23\\
97	2.17429493560255e-23\\
98	2.17429493560255e-23\\
99	2.07432349926921e-23\\
100	2.07432349926921e-23\\
};
\addplot [color=mycolor2!20, forget plot,, line width = 1.2pt]
  table[row sep=crcr]{%
12	1e-25\\
12	1\\
};
\addplot [color=mycolor2!20, forget plot,, line width = 1.2pt]
  table[row sep=crcr]{%
1	1e-7\\
400	1e-7\\
};

\end{axis}
\end{tikzpicture}%
%

\caption{Time-delay example: Relative decay of the singular values using \Cref{algo:Inter_structured_para} and structured balanced truncation.}
\label{fig:delay_SVD}
\end{figure}

\begin{figure}[!tb]
    \centering
    \begin{tikzpicture}
	\begin{customlegend}[legend columns=2, legend style={/tikz/every even column/.append style={column sep=0.25cm}} , legend entries={Orig. sys. $(n = 500)$, Algorithm 1 $(r = 12)$, Balanced truncation $(r = 12)$}, ]
	    \addlegendimage{color=black,solid,line width=1.2pt,forget plot}
	    \addlegendimage{color=purple,line width=1.3pt,dotted, mark = *, mark repeat = 5,mark options={solid}, mark size = 1pt}
	    \addlegendimage{color=green!60!black,line width=1.3pt,dashed,mark=o,mark options={solid},forget plot, mark repeat = 15}
	\end{customlegend}
    \end{tikzpicture}

    \begin{subfigure}[t]{1\textwidth}
	\centering
	\setlength\fheight{2.5cm}
	\setlength\fwidth{.75\textwidth}
    \tikzsetnextfilename{Figures/Delay_TF}%
    \input{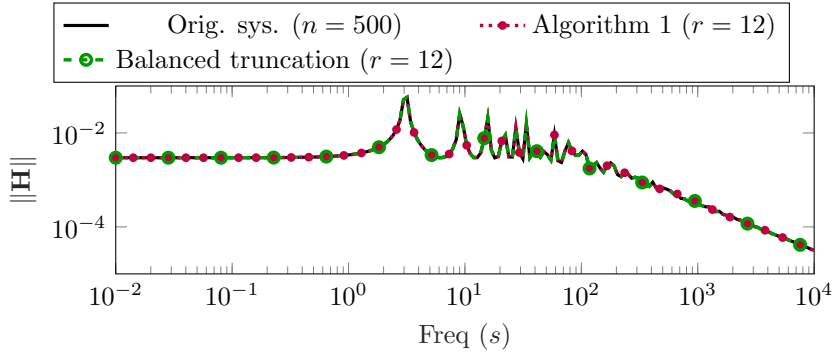}%

	\caption{The Bode plots of the original and reduced-order systems.}
    \end{subfigure}%

    \begin{subfigure}[t]{1\textwidth}
	\centering
	\setlength\fheight{2.5cm}
	\setlength\fwidth{.75\textwidth}
    \tikzsetnextfilename{Figures/Delay_TF_Error}%
    \input{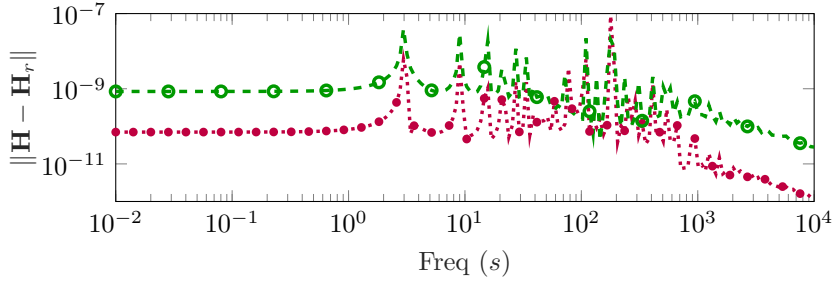}%

	\caption{The Bode plots of the errors of the reduced-order systems.}
    \end{subfigure}
    \caption{Time-delay example: The figure presents a comparison of the Bode plots of the original and reduced-order systems.}
    \label{fig:delay_exam_TF}
\end{figure}

\subsection{Heat Equation with Fading Memory} We consider the heat equation with fading memory presented in \cite{breiten2016structure}. The  example arises within the context of heat conduction in materials with fading memory. Its dynamics is governed by the following integro-PDE:
\begin{align}
\begin{split}
v_t(t,x) &= \grad v(t,x) - \int_0^t \gamma(t-s)\grad v(s,x) ds +\chi_{\omega} \bu(t) ~\quad \text{in $(0, \infty)\times \Omega$},\\
v(t,x) &= 0 \quad \text{in $(0, \infty)\times \Gamma$}, \qquad v(0,x) =0 \quad \text{in $\Omega$}, \\
v_{obs}(\cdot) &= \int_{\Omega} v(\cdot, x) dx,
\end{split}
\end{align}
where $\gamma = 1.05$, $\Omega = (0,1) \times (0,1)$, $\Gamma$ denotes the boundary of $\Omega$ and $\chi_{\omega}$ is the characteristic function of the control set $\omega = [0.15, 0.25] \times [0.2,0.3] \subset \Omega$, i.e.,
\[ \chi_{\omega}(x) :=
\begin{cases}
1, \quad \text{if $x\in\omega$},  \\ 
0, \quad \text{otherwise}.
\end{cases}
\]
As discussed in \cite{breiten2016structure}, after spatial discretization by a finite difference method, we obtain a Volterra integro-differential system whose transfer function is given by
\begin{align}
\bH(s) = \bC\left(s\bI-\bA + \frac{1}{s+ \gamma}\bA\right)^{-1}\bB.
\end{align}
We consider $128$ grid points in each direction, leading to a system of order $n=16,384$. In \Cref{fig:Heat_FM_SVD}, we first plot the relative singular values obtained using \cref{algo:Inter_structured_para} and  balanced truncation from \cite{breiten2016structure}. The figure shows that a very low-order model is possible to obtain having very high accuracy. We determine reduced-order systems of order $r = 3$ using both methods. We compare both reduced-order systems in \Cref{fig:heat_bode}, which shows that the reduced-order systems are comparable. Moreover, in \Cref{fig:heat_h2error}, we show the $\cH_2$-norm of the error systems with respect to the order of the reduced-order system. This also indicates that both methods produce reduced-order systems of very similar quality. 

\begin{figure}[!tb]
	\centering
	\definecolor{mycolor1}{rgb}{0.00000,0.44700,0.74100}%
	\definecolor{mycolor2}{rgb}{0.85000,0.32500,0.09800}%
	\definecolor{mycolor3}{rgb}{0.92900,0.69400,0.12500}%
	\begin{tikzpicture}
	\begin{customlegend}[legend columns=-1, legend style={/tikz/every even column/.append style={column sep=1cm}} , legend entries={Algorithm 1, Balanced truncation }, ]
	\addlegendimage{mycolor1, line width = 1.2pt}
	\addlegendimage{black!30!green,dashed, line width = 1.2pt}
	\end{customlegend}
	\end{tikzpicture}
	
	\setlength\fheight{3.0cm}
	\setlength\fwidth{.5\textwidth}
    \tikzsetnextfilename{Figures/Heat_FM_DecaySV}%
    \input{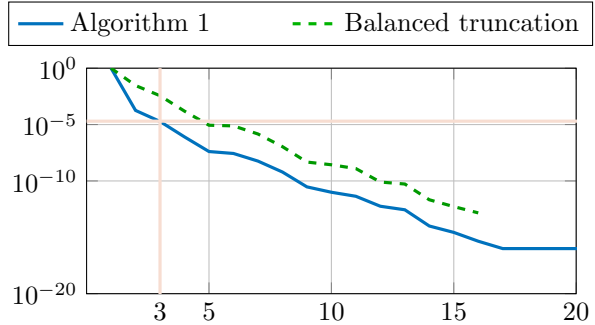}%

	\caption{Heat equation with fading memory: Relative decay of the singular values using \Cref{algo:Inter_structured_para} and structured balanced truncation.}
	\label{fig:Heat_FM_SVD}
\end{figure}

\begin{figure}[!tb]
	\centering
	\begin{tikzpicture}
	\begin{customlegend}[legend columns=2, legend style={/tikz/every even column/.append style={column sep=0.25cm}} , legend entries={Orig. sys. $(n = 16{,}384)$, Algorithm 1 $(r = 3)$, Balanced truncation $(r = 3)$}, ]
	\addlegendimage{color=black,solid,line width=1.2pt,forget plot}
	\addlegendimage{color=purple,line width=1.3pt,dotted, mark = *, mark repeat = 5,mark options={solid}, mark size = 1pt}
	\addlegendimage{color=green!60!black,line width=1.3pt,dashed,mark=o,mark options={solid},forget plot, mark repeat = 15}
	\end{customlegend}
	\end{tikzpicture}
	
	\begin{subfigure}[t]{1\textwidth}
		\centering
		\setlength\fheight{2.5cm}
		\setlength\fwidth{.36\textwidth}
		{\small %
    \tikzsetnextfilename{Figures/Heat_FM_TF}%
    \input{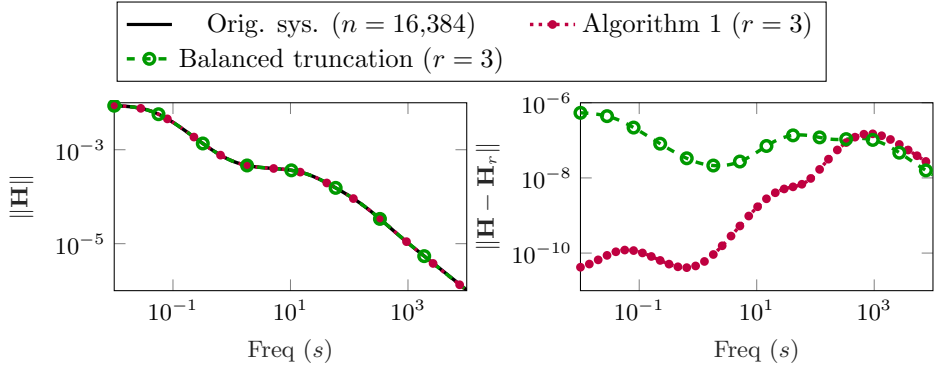}%
\hspace{-0.1cm}
    \tikzsetnextfilename{Figures/Heat_FM_TF_Error}%
    \input{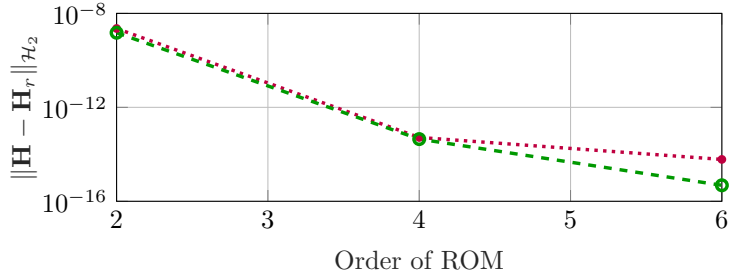}%

		}
		\caption{The Bode plots of the original, reduced-order systems, and the error system.}
		\label{fig:heat_bode}
	\end{subfigure}%
	
	\begin{subfigure}[t]{1\textwidth}
		\centering
		\setlength\fheight{2.5cm}
		\setlength\fwidth{.65\textwidth}
    \tikzsetnextfilename{Figures/Heat_FM_H2error}%
%
%
\definecolor{mycolor1}{rgb}{0.00000,0.44700,0.74100}%
\definecolor{mycolor2}{rgb}{0.85000,0.32500,0.09800}%
\begin{tikzpicture}

\begin{axis}[%
width=0.951\fwidth,
height=\fheight,
at={(0\fwidth,0\fheight)},
scale only axis,
xmin=2,
xmax=6,
xlabel style={font=\color{white!15!black}},
xlabel={Order of ROM},
ymode=log,
ymin=1e-16,
ymax=1e-08,
yminorticks=true,
grid = major,
xtick = {2,3,4,5,6}, 
ylabel style={font=\color{white!15!black}},
ylabel={$\|\bH - \bH_r\|_{\cH_2}$},
axis background/.style={fill=white},
legend style={legend cell align=left, align=left, draw=white!15!black}
]
\addplot [color=purple,line width=1.3pt,dotted, mark = *, mark repeat = 1,mark options={solid}, mark size = 1pt]
  table[row sep=crcr]{%
2	2.2759190699871e-09\\
4	5.22840448640203e-14\\
6	6.06803427986906e-15\\
};

\addplot [color=green!60!black,line width=1.3pt,dashed,mark=o,mark options={solid},forget plot, mark repeat = 1]
  table[row sep=crcr]{%
2	1.50157979166465e-09\\
4	4.44797768609839e-14\\
6	4.80382053285175e-16\\
};

\end{axis}
\end{tikzpicture}%
%

		\caption{The Bode plots of the relative errors of the reduced-order systems.}
		\label{fig:heat_h2error}
	\end{subfigure}
	\caption{Heat equation with fading memory: The figure presents a comparison of the Bode plots of the original and reduced-order systems for order $r = 3$, and $\cH_2$-norm of the error with respect to the order of the reduced-order systems.}
	\label{fig:Heat_FM_exam_TF}
\end{figure}

\subsection{Fractional Maxwell equations} Next, we consider the example mentioned in the introduction arising in computational electromagnetics, see \cite{morFenB10}. 
As described in the introduction, the transfer function of this example has a fractional integrator of the form:
\begin{equation}
\bH(s) = s\bB^T\left(s^2\bI - \frac{1}{\sqrt{s}}\bD +\bA\right)^{-1}\bB,
\end{equation}
where $n = 29,295$ is the order, $\bI,\bA$ and $\bD$ are $n \times n$ matrices and $\bB$ is a $1\times n$ matrix. This example is a limited frequency range and the interesting frequency range is $\cF:=\bbm 4e9,8e9\ebm$Hz. We aim at employing \cref{algo:Inter_structured_para} and balanced truncation.  Although the proposed balanced truncation \cite{breiten2016structure} is not designed for limited frequency range, one can integrate within $\cF$ range instead of $\bbm 0,\infty\ebm$ to determine the Gramians. However, we observe that for this example, the methodology proposed in \cite{breiten2016structure} to  obtain an approximation in low-rank factor does not converge. As discussed in \cite{breiten2016structure}, the development of low-rank solvers for Gramians of structure systems  needs some future research which is highly relevant to this. Moreover, since the example is a large-scale one, it is not possible to apply the simple \texttt{quadv} function to obtain approximations of Gramians as we did in the delay example. 

We instead compare our methodology to the method proposed in \cite{morFenB10}, where a reduced-order system is obtained using moment-matching based on a single expansion point. In order to employ \Cref{algo:Inter_structured_para}, we take $50$ points in the logarithm scale in the frequency range of interest. First, we plot the relative decay of the singular values in \Cref{fig:ElecMagnetic_SVD}. Ideally, we would like to determine a reduced-order system of order $r = 38$, which is reported in \cite{morFenB10}. However, our decay of singular values indicates that the order more than $10$ would not improve the quality of the reduced-order system as the singular values go to the level of machine precision. Therefore, we determine the reduced-order system of  order $r=8$ via our method and compare with the reduced-order system of  order $r = 38$ as constructed in \cite{morFenB10}.  In \cref{fig:ElecMagnetic_TF}, we compare both reduced-order systems. The figure suggests that our reduced-order systems outperformed the one reported in \cite{morFenB10}. Importantly, notice that this is achieved even though our reduced-order system has an order more than four times smaller than the one from \cite{morFenB10}. 

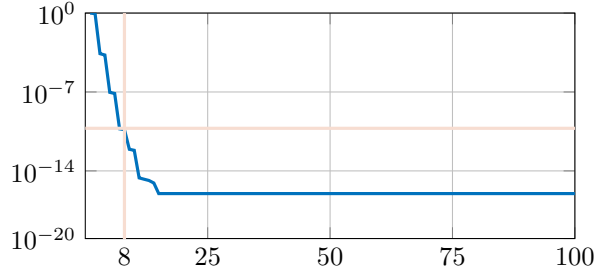
\begin{figure}[!tb]
	\centering
	\definecolor{mycolor1}{rgb}{0.00000,0.44700,0.74100}%
	\definecolor{mycolor2}{rgb}{0.85000,0.32500,0.09800}%
	\definecolor{mycolor3}{rgb}{0.92900,0.69400,0.12500}%
	
	\setlength\fheight{3.0cm}
	\setlength\fwidth{.5\textwidth}
    \tikzsetnextfilename{Figures/ElecMagnetic_DecaySV}%
%
%
\definecolor{mycolor1}{rgb}{0.00000,0.44700,0.74100}%
\definecolor{mycolor2}{rgb}{0.85000,0.32500,0.09800}%
\begin{tikzpicture}

\begin{axis}[%
width=1\fwidth,
height=\fheight,
at={(0\fwidth,0\fheight)},
scale only axis,
xmin=0,
xmax=100,
xlabel style={font=\color{white!15!black}},
ymode=log,
ymin=1e-20,
ymax=1,
yminorticks=true,
grid = major,
xtick = {8,25,50,75,100}, 
ytick = {1e0,1e-7,1e-14,1e-20},
ylabel style={font=\color{white!15!black}},
axis background/.style={fill=white},
title style={font=\bfseries},
]
\addplot [color=mycolor1, line width = 1.2pt]
  table[row sep=crcr]{%
1	1\\
2	0.884614171445349\\
3	0.000249918852978569\\
4	0.000189958064040618\\
5	9.24742998588644e-08\\
6	7.35719403485192e-08\\
7	5.39726007873654e-11\\
8	4.55454595795023e-11\\
9	8.25873681515395e-13\\
10	6.6933115671011e-13\\
11	2.40290293940316e-15\\
12	1.82934334931319e-15\\
13	1.41427247060408e-15\\
14	8.41265119399197e-16\\
15	9.7164541084774e-17\\
16	9.7164541084774e-17\\
17	9.7164541084774e-17\\
18	9.7164541084774e-17\\
19	9.7164541084774e-17\\
20	9.7164541084774e-17\\
21	9.7164541084774e-17\\
22	9.7164541084774e-17\\
23	9.7164541084774e-17\\
24	9.7164541084774e-17\\
25	9.7164541084774e-17\\
26	9.7164541084774e-17\\
27	9.7164541084774e-17\\
28	9.7164541084774e-17\\
29	9.7164541084774e-17\\
30	9.7164541084774e-17\\
31	9.7164541084774e-17\\
32	9.7164541084774e-17\\
33	9.7164541084774e-17\\
34	9.7164541084774e-17\\
35	9.7164541084774e-17\\
36	9.7164541084774e-17\\
37	9.7164541084774e-17\\
38	9.7164541084774e-17\\
39	9.7164541084774e-17\\
40	9.7164541084774e-17\\
41	9.7164541084774e-17\\
42	9.7164541084774e-17\\
43	9.7164541084774e-17\\
44	9.7164541084774e-17\\
45	9.7164541084774e-17\\
46	9.7164541084774e-17\\
47	9.7164541084774e-17\\
48	9.7164541084774e-17\\
49	9.7164541084774e-17\\
50	9.7164541084774e-17\\
51	9.7164541084774e-17\\
52	9.7164541084774e-17\\
53	9.7164541084774e-17\\
54	9.7164541084774e-17\\
55	9.7164541084774e-17\\
56	9.7164541084774e-17\\
57	9.7164541084774e-17\\
58	9.7164541084774e-17\\
59	9.7164541084774e-17\\
60	9.7164541084774e-17\\
61	9.7164541084774e-17\\
62	9.7164541084774e-17\\
63	9.7164541084774e-17\\
64	9.7164541084774e-17\\
65	9.7164541084774e-17\\
66	9.7164541084774e-17\\
67	9.7164541084774e-17\\
68	9.7164541084774e-17\\
69	9.7164541084774e-17\\
70	9.7164541084774e-17\\
71	9.7164541084774e-17\\
72	9.7164541084774e-17\\
73	9.7164541084774e-17\\
74	9.7164541084774e-17\\
75	9.7164541084774e-17\\
76	9.7164541084774e-17\\
77	9.7164541084774e-17\\
78	9.7164541084774e-17\\
79	9.7164541084774e-17\\
80	9.7164541084774e-17\\
81	9.7164541084774e-17\\
82	9.7164541084774e-17\\
83	9.7164541084774e-17\\
84	9.7164541084774e-17\\
85	9.7164541084774e-17\\
86	9.7164541084774e-17\\
87	9.7164541084774e-17\\
88	9.7164541084774e-17\\
89	9.7164541084774e-17\\
90	9.7164541084774e-17\\
91	9.7164541084774e-17\\
92	9.7164541084774e-17\\
93	9.7164541084774e-17\\
94	9.7164541084774e-17\\
95	9.7164541084774e-17\\
96	9.7164541084774e-17\\
97	9.7164541084774e-17\\
98	9.7164541084774e-17\\
99	9.7164541084774e-17\\
100	9.7164541084774e-17\\
};
\addplot [color=mycolor2!20, forget plot,, line width = 1.2pt]
  table[row sep=crcr]{%
8	1e-25\\
8	1\\
};
\addplot [color=mycolor2!20, forget plot,, line width = 1.2pt]
  table[row sep=crcr]{%
0	0.6e-10\\
400	0.6e-10\\
};
\end{axis}
\end{tikzpicture}%
%

	\caption{Fractional Maxwell equations: Relative decay of the singular values using \Cref{algo:Inter_structured_para}.}
	\label{fig:ElecMagnetic_SVD}
\end{figure}

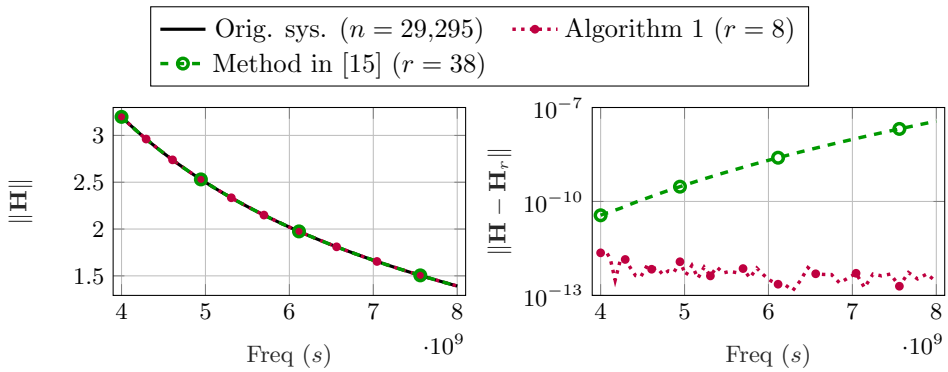
\begin{figure}[!tb]
	\centering
	\begin{tikzpicture}
	\begin{customlegend}[legend columns=2, legend style={/tikz/every even column/.append style={column sep=0.250cm}} , legend entries={Orig. sys. $(n=29{,}295)$, Algorithm 1 $(r=8)$, Method in [15] $(r=38)$}, ]
	\addlegendimage{color=black,solid,line width=1.2pt,forget plot}
	\addlegendimage{color=purple,line width=1.3pt,dotted, mark = *, mark repeat = 5,mark options={solid}, mark size = 1pt}
	\addlegendimage{color=green!60!black,line width=1.3pt,dashed,mark=o,mark options={solid},forget plot, mark repeat = 15}
	\end{customlegend}	\end{tikzpicture}
	\centering
	\setlength\fheight{2.5cm}
	\setlength\fwidth{.36\textwidth}
	{\small %
    \tikzsetnextfilename{Figures/ElecMagnetic_TF}%
%
%
\definecolor{mycolor1}{rgb}{0.00000,0.44700,0.74100}%
\definecolor{mycolor2}{rgb}{0.85000,0.32500,0.09800}%
\definecolor{mycolor3}{rgb}{0.92900,0.69400,0.12500}%
\begin{tikzpicture}

\begin{axis}[%
width=1\fwidth,
height=\fheight,
at={(0\fwidth,0\fheight)},
scale only axis,
xmin=3900000000,
xmax=8100000000,
xlabel style={font=\color{white!15!black}},
xlabel={Freq $(s)$},
ymin=1.29120588471937,
ymax=3.29822885637806,
yminorticks=true,
ylabel style={font=\color{white!15!black}},
ylabel={$\|\bH\|$},
axis background/.style={fill=white},
title style={font=\bfseries},
grid = major,
legend style={legend cell align=left, align=left, draw=white!15!black}
]
\addplot [color=black,solid,line width=1.2pt,forget plot]
  table[row sep=crcr]{%
4000000000	3.19822885634644\\
4056985547.8693	3.14959745033257\\
4114782933.90511	3.10159245897244\\
4173403723.86807	3.05420409816481\\
4232859648.28919	3.00742269864543\\
4293162604.81709	2.96123870358504\\
4354324660.59899	2.91564266622615\\
4416358054.69525	2.87062524749904\\
4479275200.52872	2.82617721363191\\
4543088688.3686	2.78228943375581\\
4607811287.84997	2.73895287749597\\
4673455950.52911	2.69615861255368\\
4740035812.4751	2.65389780226557\\
4807564196.89862	2.61216170315211\\
4876054616.8179	2.57094166243428\\
4945520777.76292	2.53022911553384\\
5015976580.51788	2.49001558353464\\
5087436123.903	2.45029267062936\\
5159913707.59567	2.41105206150519\\
5233423834.99211	2.37228551871306\\
5307981216.10943	2.33398487997497\\
5383600770.52942	2.29614205544544\\
5460297630.38399	2.25874902492358\\
5538087143.38321	2.22179783500288\\
5616984875.88664	2.18528059614601\\
5697006616.01814	2.14918947969204\\
5778168376.82537	2.11351671479324\\
5860486399.48399	2.07825458524198\\
5943977156.54781	2.04339542622261\\
6028657355.24493	2.00893162094885\\
6114543940.82116	1.97485559718681\\
6201654099.93084	1.94115982365466\\
6290005264.07595	1.90783680628651\\
6379615113.09452	1.8748790843478\\
6470501578.69826	1.84227922638516\\
6562682848.06111	1.81002982600849\\
6656177367.45842	1.77812349746704\\
6751003845.95842	1.74655287102779\\
6847181259.16584	1.7153105881182\\
6944728853.01936	1.68438929621859\\
7043666147.64263	1.65378164348507\\
7144012941.25059	1.62348027306589\\
7245789314.11126	1.59347781710407\\
7349015632.5638	1.56376689036581\\
7453712553.09426	1.53434008349604\\
7559901026.46885	1.5051899558293\\
7667602301.92664	1.47630902773843\\
7776837931.43144	1.44768977247165\\
7887629773.98482	1.41932460741098\\
8000000000	1.39120588471937\\
};

\addplot [color=purple,line width=1.3pt,dotted, mark = *, mark repeat = 5,mark options={solid}, mark size = 1pt]
  table[row sep=crcr]{%
4000000000	3.1982288563486\\
4056985547.8693	3.14959745033527\\
4114782933.90511	3.10159245897141\\
4173403723.86807	3.05420409816502\\
4232859648.28919	3.00742269864456\\
4293162604.81709	2.96123870358408\\
4354324660.59899	2.91564266622584\\
4416358054.69525	2.87062524749904\\
4479275200.52872	2.8261772136319\\
4543088688.3686	2.78228943375536\\
4607811287.84997	2.73895287749563\\
4673455950.52911	2.69615861255325\\
4740035812.4751	2.65389780226602\\
4807564196.89862	2.6121617031527\\
4876054616.8179	2.57094166243483\\
4945520777.76292	2.530229115533\\
5015976580.51788	2.49001558353455\\
5087436123.903	2.45029267062872\\
5159913707.59567	2.4110520615055\\
5233423834.99211	2.37228551871382\\
5307981216.10943	2.33398487997479\\
5383600770.52942	2.29614205544487\\
5460297630.38399	2.25874902492409\\
5538087143.38321	2.22179783500346\\
5616984875.88664	2.18528059614566\\
5697006616.01814	2.14918947969269\\
5778168376.82537	2.11351671479306\\
5860486399.48399	2.07825458524147\\
5943977156.54781	2.04339542622237\\
6028657355.24493	2.00893162094879\\
6114543940.82116	1.97485559718672\\
6201654099.93084	1.94115982365458\\
6290005264.07595	1.90783680628665\\
6379615113.09452	1.8748790843477\\
6470501578.69826	1.84227922638593\\
6562682848.06111	1.81002982600896\\
6656177367.45842	1.77812349746745\\
6751003845.95842	1.74655287102823\\
6847181259.16584	1.71531058811844\\
6944728853.01936	1.68438929621894\\
7043666147.64263	1.65378164348467\\
7144012941.25059	1.62348027306601\\
7245789314.11126	1.59347781710365\\
7349015632.5638	1.5637668903662\\
7453712553.09426	1.53434008349644\\
7559901026.46885	1.50518995582918\\
7667602301.92664	1.47630902773894\\
7776837931.43144	1.44768977247197\\
7887629773.98482	1.41932460741144\\
8000000000	1.39120588471956\\
};

\addplot [color=green!60!black,line width=1.3pt,dashed,mark=o,mark options={solid},forget plot, mark repeat = 15]
  table[row sep=crcr]{%
4000000000	3.19822885637806\\
4056985547.8693	3.14959745036932\\
4114782933.90511	3.10159245901074\\
4173403723.86807	3.05420409821042\\
4232859648.28919	3.00742269869697\\
4293162604.81709	2.96123870364454\\
4354324660.59899	2.91564266629559\\
4416358054.69525	2.87062524757949\\
4479275200.52872	2.82617721372467\\
4543088688.3686	2.78228943386233\\
4607811287.84997	2.73895287761893\\
4673455950.52911	2.69615861269537\\
4740035812.4751	2.65389780242981\\
4807564196.89862	2.61216170334144\\
4876054616.8179	2.57094166265228\\
4945520777.76292	2.53022911578349\\
5015976580.51788	2.49001558382307\\
5087436123.903	2.45029267096101\\
5159913707.59567	2.41105206188815\\
5233423834.99211	2.37228551915442\\
5307981216.10943	2.33398488048206\\
5383600770.52942	2.29614205602886\\
5460297630.38399	2.25874902559632\\
5538087143.38321	2.22179783577721\\
5616984875.88664	2.1852805970362\\
5697006616.01814	2.14918948071759\\
5778168376.82537	2.11351671597251\\
5860486399.48399	2.07825458659873\\
5943977156.54781	2.04339542778419\\
6028657355.24493	2.00893162274595\\
6114543940.82116	1.97485559925463\\
6201654099.93084	1.94115982603403\\
6290005264.07595	1.90783680902455\\
6379615113.09452	1.87487908749811\\
6470501578.69826	1.84227923001106\\
6562682848.06111	1.8100298301805\\
6656177367.45842	1.77812350226793\\
6751003845.95842	1.74655287655279\\
6847181259.16584	1.71531059447669\\
6944728853.01936	1.68438930353723\\
7043666147.64263	1.6537816519087\\
7144012941.25059	1.62348028276379\\
7245789314.11126	1.59347782826901\\
7349015632.5638	1.5637669032228\\
7453712553.09426	1.53434009830246\\
7559901026.46885	1.50518997288271\\
7667602301.92664	1.47630904738433\\
7776837931.43144	1.44768979510709\\
7887629773.98482	1.41932463349621\\
8000000000	1.39120591478582\\
};

\end{axis}
\end{tikzpicture}%
%
\hspace{-0.0cm}
    \tikzsetnextfilename{Figures/ElecMagnetic_TF_Error}%
%
%
\definecolor{mycolor1}{rgb}{0.00000,0.44700,0.74100}%
\definecolor{mycolor2}{rgb}{0.85000,0.32500,0.09800}%
\begin{tikzpicture}

\begin{axis}[%
width=1\fwidth,
height=\fheight,
at={(0\fwidth,0\fheight)},
scale only axis,
xmin=3900000000,
xmax=8100000000,
xlabel style={font=\color{white!15!black}},
xlabel={Freq $(s)$},
ymode=log,
ymin=1e-13,
ymax=1e-07,
yminorticks=true,
ylabel style={font=\color{white!15!black}},
ylabel={$\|\bH - \bH_r\|$},
axis background/.style={fill=white},
title style={font=\bfseries},
grid = major,
legend style={legend cell align=left, align=left, draw=white!15!black}
]
\addplot [color=purple,line width=1.3pt,dotted, mark = *, mark repeat = 5,mark options={solid}, mark size = 1pt]
  table[row sep=crcr]{%
4000000000	2.28604386260158e-12\\
4056985547.8693	2.69932148168465e-12\\
4114782933.90511	1.49642376218763e-12\\
4173403723.86807	3.06597433530176e-13\\
4232859648.28919	1.27006018421157e-12\\
4293162604.81709	1.40532468262056e-12\\
4354324660.59899	7.6254529178311e-13\\
4416358054.69525	4.29007163189999e-13\\
4479275200.52872	3.95118450599682e-13\\
4543088688.3686	8.03376148643984e-13\\
4607811287.84997	6.84541196685567e-13\\
4673455950.52911	7.64101967525183e-13\\
4740035812.4751	4.63387474764039e-13\\
4807564196.89862	6.05658027147689e-13\\
4876054616.8179	5.60937504891442e-13\\
4945520777.76292	1.18822263811453e-12\\
5015976580.51788	4.45339518013672e-13\\
5087436123.903	9.07257573806506e-13\\
5159913707.59567	3.55541270644021e-13\\
5233423834.99211	7.98550086939225e-13\\
5307981216.10943	4.19148632551134e-13\\
5383600770.52942	8.06482442417332e-13\\
5460297630.38399	5.2099799502918e-13\\
5538087143.38321	5.8442982578545e-13\\
5616984875.88664	5.35742602553499e-13\\
5697006616.01814	7.19831577001882e-13\\
5778168376.82537	3.61795881974465e-13\\
5860486399.48399	7.45509395609158e-13\\
5943977156.54781	4.44443044338569e-13\\
6028657355.24493	2.92346949785101e-13\\
6114543940.82116	2.2779165589924e-13\\
6201654099.93084	2.13654429932183e-13\\
6290005264.07595	1.43416292793927e-13\\
6379615113.09452	2.94967723026489e-13\\
6470501578.69826	7.7263059065811e-13\\
6562682848.06111	4.90289452796687e-13\\
6656177367.45842	4.67344747419835e-13\\
6751003845.95842	4.63167501956185e-13\\
6847181259.16584	2.50088074969178e-13\\
6944728853.01936	3.54291498142988e-13\\
7043666147.64263	5.09372271423513e-13\\
7144012941.25059	1.77492032369856e-13\\
7245789314.11126	4.80356155811055e-13\\
7349015632.5638	4.20154120553627e-13\\
7453712553.09426	4.81582371841351e-13\\
7559901026.46885	1.96479324210296e-13\\
7667602301.92664	5.37967320749453e-13\\
7776837931.43144	3.35558116376968e-13\\
7887629773.98482	4.83926176289569e-13\\
8000000000	2.725566589055e-13\\
};

\addplot [color=green!60!black,line width=1.3pt,dashed,mark=o,mark options={solid},forget plot, mark repeat = 15]
  table[row sep=crcr]{%
4000000000	3.61465218358413e-11\\
4056985547.8693	4.19991243619995e-11\\
4114782933.90511	4.44958680731669e-11\\
4173403723.86807	5.28748895803477e-11\\
4232859648.28919	6.00598607416698e-11\\
4293162604.81709	6.94943941427464e-11\\
4354324660.59899	8.11747147320157e-11\\
4416358054.69525	9.41628992054403e-11\\
4479275200.52872	1.08799715930726e-10\\
4543088688.3686	1.25256850337837e-10\\
4607811287.84997	1.44844475754706e-10\\
4673455950.52911	1.67202009243672e-10\\
4740035812.4751	1.93971901097111e-10\\
4807564196.89862	2.23956457240099e-10\\
4876054616.8179	2.58293455399299e-10\\
4945520777.76292	2.96503465571619e-10\\
5015976580.51788	3.42886979454408e-10\\
5087436123.903	3.94893024761136e-10\\
5159913707.59567	4.5637732591065e-10\\
5233423834.99211	5.26529532631693e-10\\
5307981216.10943	6.05858760404296e-10\\
5383600770.52942	6.97889103159243e-10\\
5460297630.38399	8.05375119128487e-10\\
5538087143.38321	9.27923171335486e-10\\
5616984875.88664	1.06800729901476e-09\\
5697006616.01814	1.23131750726197e-09\\
5778168376.82537	1.41731578426434e-09\\
5860486399.48399	1.63205916982992e-09\\
5943977156.54781	1.87987950389318e-09\\
6028657355.24493	2.16500527140016e-09\\
6114543940.82116	2.49297006471766e-09\\
6201654099.93084	2.87054939475758e-09\\
6290005264.07595	3.30540539933043e-09\\
6379615113.09452	3.80554310420319e-09\\
6470501578.69826	4.38250223214485e-09\\
6562682848.06111	5.04551726116183e-09\\
6656177367.45842	5.80925885160899e-09\\
6751003845.95842	6.68891357310574e-09\\
6847181259.16584	7.70182507461417e-09\\
6944728853.01936	8.86896022697902e-09\\
7043666147.64263	1.02127142124626e-08\\
7144012941.25059	1.17624458193887e-08\\
7245789314.11126	1.35473348931186e-08\\
7349015632.5638	1.56061540048967e-08\\
7453712553.09426	1.79788601573401e-08\\
7559901026.46885	2.07144418340444e-08\\
7667602301.92664	2.38709701654202e-08\\
7776837931.43144	2.75117953200304e-08\\
7887629773.98482	3.1713751026983e-08\\
8000000000	3.65638296696173e-08\\
};

\end{axis}
\end{tikzpicture}%
%

	}
	\caption{Fractional Maxwell equations: The Bode plots of the original, reduced-order systems, and the error system.}
	\label{fig:ElecMagnetic_TF}
\end{figure}

\subsection{Parametric Anemometer} An anemometer, also known as a thermal mass flow meter, has sensors, namely heater and temperature after and before the heater in the direction of the flow as shown in \Cref{fig:Aneomometer_model}. Due to the circulation of the flow, a temperature difference occurs between the sensors. Measuring the temperature difference allows us to estimate the fluid flow, for more details see \cite{miller1983flow}. 
\begin{figure}[!tb]
\centering
 \includegraphics[width = 0.8\textwidth]{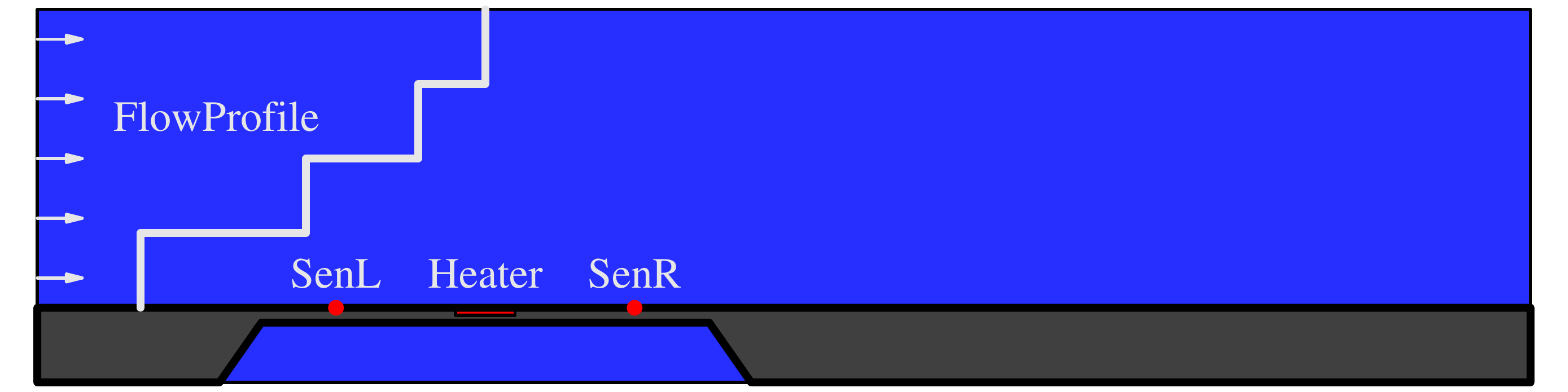}
   \caption[Schematic diagram of a two-dimensional anemometer]{Schematic diagram of a two-dimensional anemometer, cf. \cite{morwiki_anemom}.}
   \label{fig:Aneomometer_model}
\end{figure}

The dynamics of the anemometer is governed by the convention-diffusion PDE as follows:
\begin{equation}
 \rho c\dfrac{\partial T}{\partial t} = \nabla \cdot \left(\kappa \nabla T\right) - \rho c v \nabla T + \dot q,
\end{equation}
where $\rho$ represents the mass density; $c$, $\kappa$, $v$ are the specific heat,  thermal conductivity, and fluid velocity, respectively. Moreover, $\dot q$ denotes the heat flow into the system caused by the heater. We set $\rho =1$ and consider all other fluid properties as parameters. A discretization of the PDE leads to a parametric system whose transfer function is given as follows:
\begin{equation*}
 \bH(s,\bp) = \bC\left(s\left(\bE_0 + p_1 \bE_1\right) - \left(\bA_0 + p_2 \bA_1 + p_3 \bA_2\right)\right)^{-1}\bB,
\end{equation*}
where $\bp \in \R^3$ and $p_i$ denotes the $i$th component of the vector $\bp$, which is given in terms of the fluid properties as follows:
\begin{equation*}
 \bp = \bbm p_1\\p_2\\p_3\ebm = \bbm c \\ \kappa \\ cv \ebm.
\end{equation*}
The output matrix $\bC$ is chosen such that it yields the output as the difference between the two sensors. For more details, we refer to \cite{moosmann2005parameter} and for the model, to \cite{morwiki_anemom}. A finite-element discretization yields a parametric model of order $n = 29,008$, and the model essentially has three parameters. 

In order to apply the proposed method, we take $300$ points for frequency $s$ in the logarithmic way and $300$ random points for the parameters. For frequency, we consider the range $\bbm 10^1,10^5\ebm$, and the parameters are generally considered in the intervals: $c \in \bbm 0,1\ebm$, $\kappa \in \bbm 1,2\ebm$, and $v \in \bbm 0.1,2\ebm$. First, in \Cref{fig:Anemometer_SVD}, we plot the singular values, obtained by employing \Cref{algo:Inter_structured_para}, which exhibits a rapid decay. We now determine reduced-order systems of order $r = 78$, meaning that we consider the singular vectors corresponding to the relative singular values up-to $10^{-6}$. We compare the Bode plots of the original and reduced-order systems in \Cref{fig:Anemometer_TF_bode} for two different parameter values, which clearly match for both parameters all frequencies. 

Moreover, in \cref{fig:Anemometer_exam_TF}, we plot the Bode diagram of the error systems, the difference of the original and reduced-order systems, for different parameter configurations. The obtained reduced-order system captures the dynamics of the original systems very well.   From the figures, it can be seen that  the error is below $2\cdot10^{-7}$ for all frequencies and four considered parameter configurations.

\begin{figure}[!tb]
\centering
\setlength\fheight{3.0cm}
\setlength\fwidth{.5\textwidth}
{\small %
    \tikzsetnextfilename{Figures/Anemometer_DecaySV}%
    \input{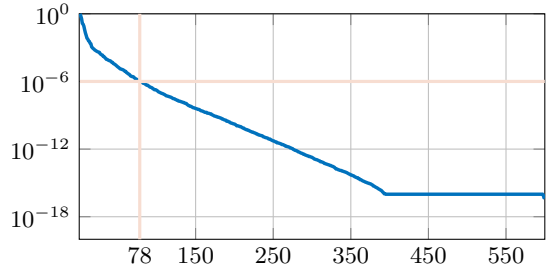}%
}
\caption{Anemometer example: relative decay of the singular values using \Cref{algo:Inter_structured_para}.}
\label{fig:Anemometer_SVD}
\end{figure}

\begin{figure}[!tb]
	\begin{tikzpicture}
	\begin{customlegend}[legend columns=2, legend style={/tikz/every even column/.append style={column sep=0.25cm}} , legend entries={Orig. sys. $(n = 29{,}008)$, Algorithm 1 $(r = 78)$, Balanced truncation $(r = 3)$}, ]
	\addlegendimage{color=blue, forget plot,line width = 1.3pt}
	\addlegendimage{color=green, dashed,line width = 1.3pt}
	\end{customlegend}
	\end{tikzpicture}
    \centering
	\setlength\fheight{2.5cm}
	\setlength\fwidth{.75\textwidth}
    \tikzsetnextfilename{Figures/Anemometer_TF}%
    \input{Figures/Anemometer_TF.tikz}%

	\caption{Anemometer example: The Bode plots of the original and reduced-order systems for parameter values $\bp^{(1)} : \left(0,1,0.1\right)$ and $\bp^{(2)}:\left(0.67,1.67,1.37\right)$.}
	\label{fig:Anemometer_TF_bode}
\end{figure}
\begin{figure}
	\centering
	\setlength\fheight{2.5cm}
	\setlength\fwidth{.75\textwidth}
    \tikzsetnextfilename{Figures/Anemometer_TF_Error}%
    \input{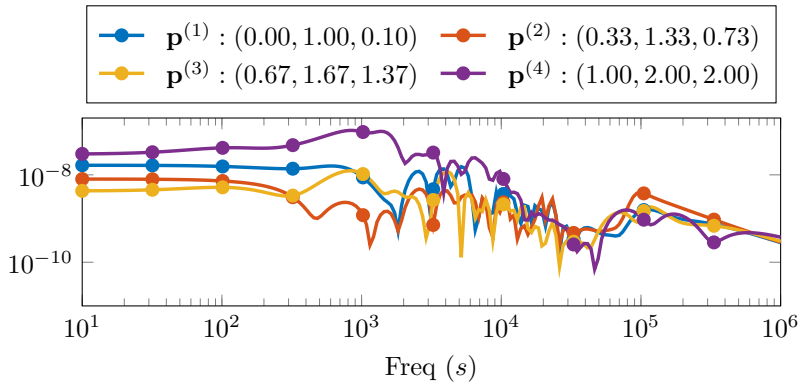}%

	\caption{Anemometer example: The Bode plots of the absolute error between the original and the reduced-order systems for four different parameter values.}
    \label{fig:Anemometer_exam_TF}
\end{figure}

\subsection{Parametric Butterfly Gyroscope}
As the last example, we consider a parametric butterfly gyroscope example.  The butterfly is a vibrating micro-mechanical gyro, which  measures angular rates in up to three axes. It  is used in inertial navigation applications. The model has mainly two parameters of interest: the rotation velocity $\theta$ around the $x$-axis and the width of bearing $d$ as shown in \Cref{fig:Gyro_diagram}. The system and its model reduction problem  has been extensively  studied in \cite{morMoo07}. A finite element discretization leads to a parametric model  for the gyroscope of the form:
\begin{figure}[!b]
	\centering
	\includegraphics[height = 3cm]{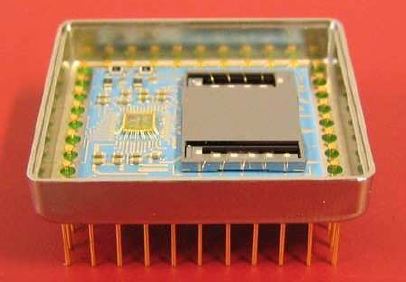}\hspace{0.5cm}
	\includegraphics[height = 3cm]{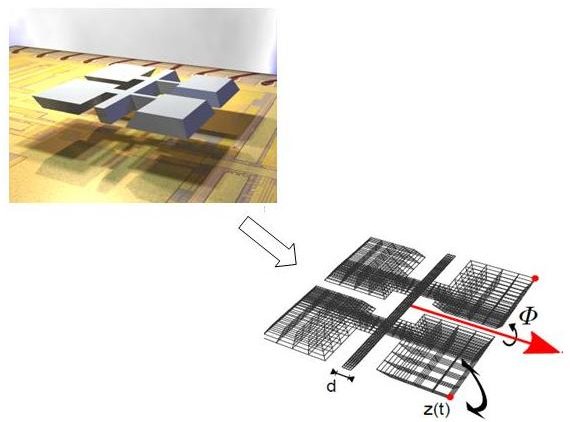}
	\caption{Semantic gyroscope diagram \cite{morwiki_modgyro}.}
	\label{fig:Gyro_diagram}
\end{figure}
\begin{equation}
\begin{aligned}
\bM(d) \ddot \bx(t) + \bD(d,\theta) \dot \bx(t) + \bK(\theta) &= \bB \bu(t),\\
\by(t) &= \bC \bx(t),
\end{aligned}
\end{equation}
where $\bM(d)  = \bM_1 + d\bM_2 \in \Rn, 
\bD(d,\theta) = \theta \left(\bD_1 + d \bD_2\right),\in \Rn$,
$\bK(d) = \bT_1 + \frac{1}{d}\bT_2 + d\bT_3,\in \Rn$, $\bB \in \Rn$, and $\bC^T \in \Rn$; $\bx(t)\in \Rn$, $\bu(t)\in \Rm$, and $\bu(t)\in \Rq$ are the state, input and output vectors, respectively. Typical ranges for the parameters $\theta$ and $d$ are $\bbm 10^{-5},10^{-7}\ebm$ and $\bbm 1,2\ebm$, respectively. Finite-element discretization leads to $ n = 17,913$. Normally, the system is operated in the frequency range  $2\pi\cdot\bbm 0.025,0.25\ebm$. For more details on the model, we refer the reader to \cite{morMoo07,morwiki_modgyro}.

\begin{figure}[!tb]
\centering
\setlength\fheight{3.0cm}
\setlength\fwidth{.5\textwidth}
{\small
    \tikzsetnextfilename{Figures/Gyro_DecaySV}%
    \input{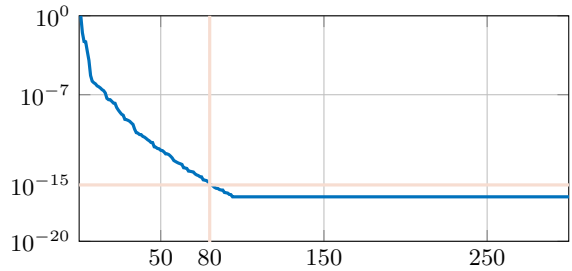}%

}
\caption{Gyro example: relative decay of the singular values obtained using \Cref{algo:Inter_structured_para}.}
\label{fig:Gyro_SVD}
\end{figure}

\begin{figure}[!tb]
    \centering
	\definecolor{mycolor1}{rgb}{0.00000,0.44700,0.74100}%
	\definecolor{mycolor2}{rgb}{0.85000,0.32500,0.09800}%
	\definecolor{mycolor3}{rgb}{0.92900,0.69400,0.12500}%
        \begin{tikzpicture}
	\begin{customlegend}[legend columns=2, legend style={/tikz/every even column/.append style={column sep=0.250cm}} , legend entries={Orig. sys. $(n=17{,}913)$, Algorithm 1 $(r=80)$, Method in [14] $(r=210)$}, ]
	    \addlegendimage{color=mycolor1,solid,line width=1.2pt}
	    \addlegendimage{color=mycolor2,line width=1.3pt,dashed}
	    \addlegendimage{color=green!70!black,line width=1.3pt,dotted}
	\end{customlegend}
    \end{tikzpicture}
        \begin{subfigure}[t]{1\textwidth}
	\centering
	\setlength\fheight{4.5cm}
	\setlength\fwidth{.75\textwidth}
	{\small 
    \tikzsetnextfilename{Figures/Gyro_TF}%
    \input{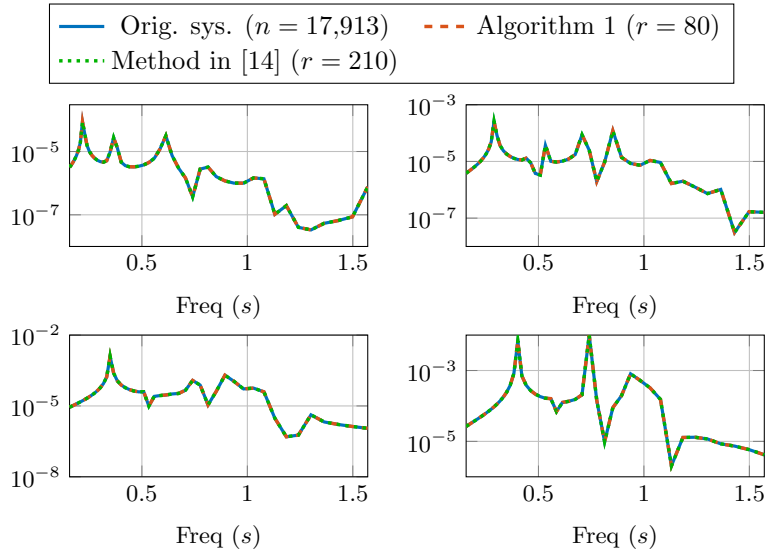}%

	}
	\caption{The Bode plots of the original and reduced-order systems for parameter values $\bp^{(1)}: \left(1.00,10^{-7}\right)$ at top-left, $\bp^{(2)}:\left(1.33,4.64\cdot 10^{-7}\right)$ at top-right, $\bp^{(3)}: \left(1.67,2.15\cdot 10^{-6}\right)$ at bottom-left and $\bp^{(4)}:\left(2.00,10^{-5}\right)$ at top-right.}
	\label{fig:Gyro_TF_bode}
\end{subfigure}
    \begin{subfigure}[t]{1\textwidth}
	\centering
	\setlength\fheight{4.5cm}
	\setlength\fwidth{.75\textwidth}
	{\small
    \tikzsetnextfilename{Figures/Gyro_TF_Error}%
    \input{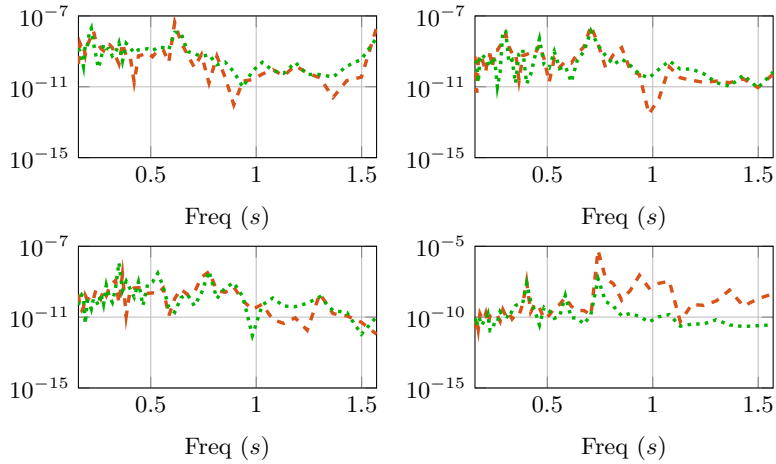}%

	}
	\caption{The Bode plots of the absolute error between the original and reduced-order systems for the above considered parameter values.}
    \label{fig:Gyro_exam_TF}
\end{subfigure}
\caption{Gyro example: Comparison of the original and reduced-order systems for different parameters values.}
\label{fig:gyro_plot}
\end{figure}

In order to apply the proposed method, we take $500$ points for frequency $s$ in the logarithmic way and the same number of random points for parameter $\bp = \bbm d,\theta\ebm^T$ in the considered range.  In \Cref{fig:Anemometer_SVD}, we first show the singular values, obtained by employing \Cref{algo:Inter_structured_para}, which indicates a rapid decay. Since the magnitude of the transfer function of the system is very small and wide range $~(10^{-7}{-}10^{-3})$, we choose to truncate at a relatively lower level. Hence, we truncate at  $10^{-15}$, thus leading to a reduced-order system of order $r = 80$. We compare the quality of the reduced-order system with the reduced-order system, obtained in \cite{morFenAB17}, where the authors have obtained a reduced-order system of order $r = 210$.  

We compare the Bode plots of the original and reduced-order systems in \Cref{fig:Anemometer_TF_bode} for four different parameter settings and the Bode plots of the error systems are plotted in \Cref{fig:gyro_plot}. These figures indicate that both reduced-order systems are of very similar quality; but our reduced-order system is of order $r = 80$, whereas the method proposed in \cite{morFenAB17} yields the reduced-order system of order $r = 210$, which is more than two-and-half times larger than ours.

\section{Conclusions}
In this paper, we have studied model order reduction for linear structured parametric systems.  Firstly, we recall the construction of an interpolatory reduced-order system for a given set of interpolation points. Then, we have defined the concepts of reachability and observability for  linear structured parametric systems and connected them with the interpolation-based MOR methods. Hence, by  combining both features, we discuss the construction of reduced-order systems keeping the subspaces that are the most reachable and observable simultaneously. Moreover, we have shown the efficiency of the proposed methods by means of various examples, appearing in science and engineering. 

As future directions, the notion of minimal structured realizations opens several future directions, in particular, construction of minimal realization via Petrov-Galerkin projection. One interesting future direction would be to combine the knowledge of error estimates, e.g. from \cite{morFenAB17} that allows us to choose good interpolations instead of just taking them randomly or in a logarithmic scale.  Moreover, an extension of structured nonlinear systems would be of highly relevant contribution to the reduced-order modeling community by combine the ideas presented in this paper and in \cite{morBenG19}. 

\section*{Acknowledgements}
We would like to thank Ali Seyfi for helping us in the implementations of numerical results during his internship at Max Planck Institute for Dynamics of Complex Technical Systems, Magdeburg, Germany. Moreover, we would like to express our gratitude to Dr.\ Tobias Breiten and Dr.\ Lihong Feng for providing the data from their publications. 
\bibliographystyle{siamplain}
\bibliography{mor,igorBiblio}

\end{document}